\documentclass[11pt, a4paper, reqno]{amsart}
\usepackage{amsmath,amsfonts, amsthm,amssymb,amsfonts,amscd, graphicx, hyperref, enumerate,mathtools, mathrsfs, dsfont, upgreek, appendix}

\newtheorem{theorem}{Theorem}[section]
\newtheorem{corollary}[theorem]{Corollary}
\newtheorem{lemma}[theorem]{Lemma}
\newtheorem{definition}[theorem]{Definition}
\newtheorem{prop}[theorem]{Proposition}
\theoremstyle{remark}
\newtheorem{remark}[theorem]{Remark}
\usepackage[left=2.5cm, right=2.5 cm]{geometry}

\newcommand{\eps}{\varepsilon}
\makeatletter

\hypersetup{
  colorlinks   = true, %Colours links instead of ugly boxes
  urlcolor     = black, %Colour for external hyperlinks
  linkcolor    = black, %Colour of internal links
  citecolor    = black %Colour of citations
}

\mathtoolsset{showonlyrefs=true}

\newcommand{\norm}[1]{\lVert#1\rVert} % norm
\DeclareMathOperator*{\esssup}{ess.\,sup} %essential supremum
\AtBeginDocument{
\let\div\undefined\DeclareMathOperator{\div}{div} %divergence
}
\AtBeginDocument{
}
\makeindex

\author[Colombo]{Maria Colombo}
\address{EPFL SB, Station 8, 
CH-1015 Lausanne, Switzerland
}
\email{maria.colombo@epfl.ch}

\author[Haffter]{Silja Haffter
}
\address{EPFL SB, Station 8, 
CH-1015 Lausanne, Switzerland
}
\email{silja.haffter@epfl.ch}

\begin{document}
\title[Global regularity for Navier-Stokes below the critical order]{Global regularity for the hyperdissipative Navier-Stokes equation below the critical order}%Global regularity for the slightly supercritical Navier-Stokes equation
\maketitle
%\tableofcontents

\begin{abstract}
We consider solutions of the Navier-Stokes equation with fractional dissipation of order $\alpha\geq 1$. We show that for any divergence-free initial datum $u_0$ such that $\norm{u_0}_{H^{\delta}} \leq M$, where $M$ is arbitrarily large and $\delta$ is arbitrarily small, there exists an explicit $\epsilon=\epsilon(M, \delta)>0$ such that the Navier-Stokes equations with fractional order $\alpha$ has a unique smooth solution for $\alpha \in (\frac{5}{4}-\epsilon, \frac{5}{4}]$. 
%This follows from a new stability result on smooth solutions of the hyperdissipative Navier-Stokes equations showing that the set of initial data and fractional orders giving rise to smooth solutions is open in $H^{s} \times [1, \frac{5}{4}]$ for any $s > \frac{1}{2}$.
This is related to a new stability result on smooth solutions of the Navier-Stokes equations with fractional dissipation showing that the set of initial data and fractional orders giving rise to smooth solutions is open in $H^{5/4} \times (\frac 34, \frac{5}{4}]$. 
\end{abstract}

\textbf{Keywords:} Navier-Stokes equations, fractional dissipation, global regularity, supercritical equation.

\textbf{2010 Mathematics Subject Classification:} 76D05, 35Q35, 35B65.

\section{Introduction}

For $\alpha>0$ we consider the Cauchy problem for the Navier-Stokes equations with fractional dissipation (of order $\alpha$) in three space dimensions 
\begin{equation}\label{eq: NSalpha}
\begin{cases} \partial_t u + (u\cdot \nabla) u + \nabla p = - (-\Delta)^\alpha u  \\
\div u =0% \\u(\cdot, 0)= u_0 
\, .
\end{cases}
\end{equation}
Here $u: \mathbb{R}^3 \times [0,+\infty) \to \mathbb{R}^3$ is the velocity of an incompressible fluid, $p: \mathbb{R}^3 \times [0,+\infty) \to \mathbb{R}$ is the associated hydrodynamic pressure, $u( \cdot, 0)= u_0$ is a given, divergence-free initial datum, and the operator $(- \Delta)^\alpha$ corresponds to the Fourier symbol $|\xi|^{2\alpha}$.
%{\color{blue} SUGGESTION:}
The natural a priori bound associated with system \eqref{eq: NSalpha} is given by the total energy 
$$\mathcal{E}(u; t) := \frac{1}{2} \int \lvert u \rvert^2 (x,t) \, \mathrm{d}x + \int_0^t \int \lvert (-\Delta)^{\alpha/2} u \rvert^2 (x,\tau) \, \mathrm{d}x \, \mathrm{d}\tau \,.$$
 Moreover, the system  \eqref{eq: NSalpha} has a natural scaling which preserves the equation: namely, given any solution $(u,p)$, also $(u_r, p_r)=(r^{2\alpha-1}u(rx,r^{2\alpha}t), r^{4\alpha-2}p(rx,r^{2\alpha}t))$ is a solution to the same equation. %Correspondingly, the local energy scales like $\| u_r\|_{L^\infty([0,\infty), L^2)} = \frac 1 {r^{5-4\alpha}} \| u\|_{L^\infty([0,\infty), L^2)} $.
Correspondingly, the total energy of the rescaled solution scales like $\mathcal{E}(u_r; t)= \frac{1}{r^{5-4\alpha}} \mathcal{E}(u;t)$. Consequently, the equation is called critical if the energy is scaling-invariant, namely for $\alpha= \frac 54$, subcritical for $\alpha>\frac 54$ and supercritical for $\alpha<\frac 54$.

Our main result shows that, given any (possibly large) initial datum $u_0$, the supercritical Navier-Stokes equation is globally well-posed at least for an open interval of orders below $\frac 54$. 
\begin{theorem}[Global regularity below the critical order]\label{thm:main1} Let $\delta \in (0,1]$ arbitrary. Then for any divergence-free $u_0 \in H^{\delta}$ with $\norm{u_0}_{H^{\delta}}\leq M$ there exists $\epsilon= \epsilon(M, \delta)>0$ such that the Navier-Stokes equations of fractional order $\alpha \in (\frac{5}{4}-\epsilon, \frac{5}{4}]$ has a unique global smooth solution starting from $u_0$. Moreover, the dependence of $\epsilon$ on $M$ and $\delta$ is explicit through \eqref{eq: epsilonexplicit 54 ugly}.%{\color{red} $\delta=0$}
\end{theorem}

This result is related to a more general stability result on smooth solutions of the hyperdissipative Navier-Stokes equations, with respect to variations of both the initial datum and the fractional order. The following theorem quantifies the convergence of the initial data in a stronger norm since at difference from Theorem~\ref{thm:main1} it covers also ipodissipative orders.

\begin{theorem}\label{thm:main2} Let $p \in [1,2)$. The set of initial data $u_0$ and fractional orders $\alpha$ giving rise to global smooth solutions in $C([0,+\infty), H^1)$ of the fractional Navier-Stokes equations is an open set in 
$$\left\{ u_0 \in L^p\cap H^{\frac 54}(\mathbb{R}^3; \mathbb{R}^3) : \div u_0=0\right\} \times \Big(\frac{3}{4}, \frac{5}{4}\Big]\, ,$$ endowed with the product topology.%, for any $s\geq  \frac{7}{6}$.
%{\color{red} add second statement between 3/4 and 1}
\end{theorem}

When $\alpha \geq \frac 54$ and {$u_0$ is smooth with sufficient decay at infinity}, the existence of global smooth solutions is well-known since \cite{Lions}.  The attempt to build global smooth solutions for supercritical {Navier-Stokes} equations has been widely pursued. It has been done successfully for small initial data in scaling invariant norms, as in the classical results of Kato for $\alpha=1$. % and in the work by ... to require the minimal necessary regularity on the initial datum.

In a different spirit, \cite{Tao} proved that the existence of a global regular solution still holds for any sufficiently regular initial datum when the right-hand side of the first equation in \eqref{eq: NSalpha} is replaced by a logarithmically supercritical operator; later, this result was generalized with a Dini-type condition in \cite{BMR}.
Recently, a result similar to Theorem~\ref{thm:main1} was obtained in \cite{CDM}, showing that for any $H^1$ initial datum there
are global smooth solutions of \eqref{eq: NSalpha} whenever $\alpha$ is sufficiently close to $\frac 54$ (and this closeness is uniform on bounded subsets of $H^1$). This was proven by means of an $\epsilon$-regularity result on a suitable notion of weak solutions, the known global regularity at $\alpha=\frac 54$ and a compactness argument. The present paper provides a simpler proof with respect to \cite{CDM}, and, not relaying on any contradiction argument, it has the advantage to provide an explicit $\eps$ depending only on the size of the initial datum.

The paper is organized as follows. After recalling different notions of weak solutions  to the system \eqref{eq: NSalpha} in Section~\ref{sec:prelim}, we prove a stability result on any finite time interval (of arbitrary length) in the fractional order and with respect of variations of the initial datum in Section~\ref{prop: locstab H3}. This estimate holds for any $\alpha \in (0,\frac32)$, but for low $\alpha$ we require stronger norms in the convergence of initial data. In Section~\ref{sec:leray}, we use the dissipation to pass from  a local stability to a global result. In particular, following the ideas of Leray~\cite{Leray} (recently revisited in \cite{JiuWang} to cover the equation with fractional dissipation) we show the eventual regularization of the Navier-Stokes equation for $\alpha > \frac 56$. In turn, this kind of argument breaks down at $\alpha=\frac56$ since both norms $\|u(\cdot,t)\|_{L^2}$ and $\|(-\Delta)^{\frac\alpha2}u(\cdot,t)\|_{L^2}$ become scaling critical at this exponent.
Hence, we answer in Section~\ref{sec:reg-suit} an open question in \cite{JiuWang} by showing that, even for $\alpha\in (\frac34, \frac56)$, the eventual regularization holds, relying this time on partial regularity arguments.

%Cite and improve with suitable weak solutions...
%\end{itemize

A result of the type of Theorem~\ref{thm:main1} was recently obtained for the supercritical nonlinear wave euqation in \cite{CH}, the SQG equation in \cite{Cotivicol} and has been generalized to other active scalar equations \cite{SGS}; we expect that also a stability result as Theorem~\ref{thm:main2} could be suitably adapted to their context.
%The stability on finite intervals in time is obtained through energy estimates. 
The long-time regularity relies in the case of the SQG equation on the scalar nature of the equation and on the maximum principle, indeed it works for any fractional dissipation $\alpha \in (0,1)$; a similar argument does not appear to apply to the Navier-Stokes equations.

\section{Preliminaries}\label{sec:prelim}
%\subsection{Notation} 
%{\color{red} generic n or 3 and no dependence?}
%\subsection{Local existence of smooth solutions} We recall that we can build strong solutions for a short time provided that the initial datum is sufficiently regular for any $\alpha>0\, .$
%
%\begin{theorem}[\cite{BM}, Theorem 3.4.]\label{thm: localexist} Let $m\geq 3$ and $u_0 \in H^m(\mathbb{R}^3)$ divergence-free. Then there exists a time $T>0$ with the rough upper bound $T \leq \frac{1}{c_m \norm{u_0}_{H^m}}$, such that there exists a unique solution $u\in C([0,T], C^2) \cap C^1([0,T], C^1)$ to \eqref{eq: NSalpha} with the bound
%\begin{equation*}
%\sup_{t\in [0,T]} \norm{u(t)}_{H^m} \leq  \frac{\norm{u_0}_{H^m}}{1-c_mT \norm{u_0}_{H^m}} \, .
%\end{equation*}
%\end{theorem}
%The theorem in \cite{BM} is stated for % the classical Navier-Stokes 
%$\alpha=1$, but with arbitrary viscosity $\nu \geq 0$. In particular, it covers also the Euler equations. The proof, which is based on energy methods, can easily be adapted to the fractional Navier-Stokes equations of order $\alpha>0$. Thanks to the uniqueness, we may continue the local solution $u$ until a maximal time of existence $T_{max} \in \mathbb{R}_+ \cup \{ + \infty \}$ such that $u\in C([0,T_{max}), H^m)$ and %. Moreover, by a standard argument, if $T_{max}< + \infty$, we must have 
%\begin{equation}\label{eq: blowup}
%\limsup_{t \uparrow T_{max}} \norm{u(t)}_{H^m} = + \infty \qquad \mbox{if $T_{max}< + \infty$}\,.
%\end{equation}

\subsection{Leray-Hopf solutions} We recall the classical concept of Leray-Hopf solutions introduced in the seminal papers of Leray \cite{Leray} and Hopf \cite{Hopf}.

\begin{definition}Let $u_0 \in L^2(\mathbb{R}^3)$ divergence-free. A Leray-Hopf solution is a distributional solution  $(u,p)$ of \eqref{eq: NSalpha} on $\mathbb{R}^3 \times (0,T)$ such that 
\begin{enumerate}[i)]
\item $u\in L^\infty((0,T), L^2(\mathbb{R}^3)) \cap L^2((0,T), H^\alpha(\mathbb{R}^3))$,
\item $p$ is the potential-theoretic solution of $-\Delta p = \div \div u\otimes u$,
\item For every $t\in (0,T)$, for $s=0$ and for almost every $0<s<t$ there holds the global energy inequality
\begin{align}%\label{eq: energyineq}
%\frac{1}{2} \int \lvert u \rvert^2 (x,t) \, \mathrm{d}x + \int_0^t \int \lvert (-\Delta)^{\alpha/2} u \rvert^2(x, \tau) \, \mathrm{d}x  \, \mathrm{d}\tau &\leq \frac{1}{2} \int \lvert u_0 \rvert^2 (x) \, \mathrm{d}x \\ 
\label{eq: energyineq2}
\frac{1}{2} \int \lvert u\rvert^2 (x,t) \, \mathrm{d}x + \int_s^t \int \lvert (-\Delta)^{\alpha/2}  u \rvert^2(x, \tau) \, \mathrm{d}x  \, \mathrm{d}\tau &\leq \frac{1}{2} \int \lvert u \rvert^2 (x,s)\ \, \mathrm{d}x \, .
\end{align}
\end{enumerate}
\end{definition}
From \eqref{eq: energyineq2} we deduce that up to changing $u$ on a set of measure $0$, we have $u\in C([0,T),L^2_w)$.
\begin{theorem}[Existence of Leray solutions]\label{thm: existLerayHopf} Let $\alpha>0$ and $u_0 \in L^2$ divergence-free. Then there exists a Leray-Hopf solution on $\mathbb{R}^3 \times (0,+\infty)$ to  \eqref{eq: NSalpha}.
\end{theorem}
Let us recall from \cite{CDM, TY} that in the fractional case a Leray-Hopf solution can still be constructed following Leray's strategy in \cite{Leray}, that is as limit of solutions of the regularized system
\begin{equation}\label{eq: regularizedNS}
\begin{cases}
\partial_t u + ((u \ast \varphi_\epsilon) \cdot \nabla) u + \nabla p = -(-\Delta)^\alpha u \\
\div u =0 %\\
%u(\cdot, 0)=u_0
 \, ,
\end{cases}
\end{equation} with the same initial datum $u_0$, 
where $\{\varphi_\epsilon \}_{\epsilon>0}$ is a family of mollifiers in space. Indeed, \eqref{eq: regularizedNS} admits a unique solution $u_\epsilon \in C([0,+\infty), L^2) \cap L^2_{loc}([0,+\infty), \dot{H}^\alpha)$ which is smooth in the interior and satisfies the local and global energy equality associated to \eqref{eq: regularizedNS}. The pressure $p_\epsilon$ can be assumed to be the potential-theoretic solution of
\begin{equation}
-\Delta p_\epsilon = \div  \big(((u_\epsilon \ast \varphi_\epsilon) \cdot \nabla) u_\epsilon\big) \, .
\end{equation}
Then there exists $u\in L^\infty([0,+\infty), L^2) \cap L^2([0,+\infty), \dot{H}^\alpha)$ such that for any $2 \leq p<\frac{6+4\alpha}{3}$ we have $u_\epsilon \rightarrow u$ strongly in $L^p_{loc}([0,+\infty) \times \mathbb{R}^3)$ and $p_\epsilon \rightarrow p$ strongly in $L^{p/2}_{loc}([0,+\infty) \times \mathbb{R}^3)$, where $p$ is now the potential theoretic solution of $-\Delta p=\div \div u\otimes u$. The pair $(u, p)$ is a Leray-Hopf solution of \eqref{eq: NSalpha}. Moreover, since $u_\epsilon$  satisfies even the local energy equality, the obtained weak solution $(u,p)$ is in fact even a suitable weak solution.

A point $(x, t)\in \mathbb{R}^3 \times (0,+\infty)$ is called a regular point of a Leray-Hopf %or suitable weak 
solution $(u,p)$  if there is a cylinder $Q_r(x, t)$ where $u$ is continuous. We denote by $Sing(u)$ the (relatively closed) set of points which are not regular. By classical boot-strap methods, we know that if $\alpha>\frac{1}{2}$ and $u \in L^\infty(Q_r(x,t))$ for some $r>0$, then $(x,t)$ is regular.

\subsection{Suitable weak solutions} Suitable weak solutions for the classical Navier-Stokes system have been introduced by Scheffer \cite{Scheffer1, Scheffer2} and Caffarelli-Kohn-Nirenberg \cite{CKN}. Only recently, the concept has been adapted to the ipodissipative range $\alpha \in (0,1)$ in \cite{TY} and to the hyperdissipative range $\alpha \in (1,2)$ in \cite{CDM}. The main ingredient is the Caffarelli-Silvestre extension for the fractional Laplacian \cite{CS} which allows to write a localized energy inequality also in the non-local setting. The existence of suitable weak solutions is obtained again through the regularization \eqref{eq: regularizedNS}, as shown in \cite{TY,CDM}. We recall here the notion of suitable weak solution for $\alpha \in (0,1)$, which will be essential to show the eventual regularization of solutions in Section~\ref{sec:reg-suit}.  
\begin{definition}Let $\alpha \in (0,1)$. A Leray-Hopf solution $(u,p)$ on $\mathbb{R}^3 \times (0,T)$ is a suitable weak solution, if for every $\varphi\in C^\infty_c(\mathbb{R}^4_+\times (0,T))$ with $\partial_y \varphi(\cdot, 0, \cdot)=0$ on $\mathbb{R}^3 \times (0,T)$ and $\varphi \geq 0$ and almost every $t\in (0,T)$, there holds the localized energy inequality
\begin{align}\label{eq: localized energy inequality ipo}
\frac{1}{2}\int_{\mathbb{R}^3} &\lvert u \rvert^2 (x,t)  \varphi(x,0,t) \, \mathrm{d}x+  c_\alpha \int_0^t \int_{\mathbb{R}^4_+} y^b \lvert \overline{\nabla} u^\ast \rvert^2 \varphi \, \mathrm{d}x \, \mathrm{d}y \, \mathrm{d}\tau  \\
&\leq \int_0^t \int_{\mathbb{R}^3} \left(\frac{\lvert u \rvert^2}{2} \partial_t \varphi \vert_{y=0} + \left( \frac{\lvert u \rvert^2}{2}+ p\right) u\cdot \nabla \varphi\vert_{y=0} \right)\, \mathrm{d}x \, \mathrm{d}\tau +c_\alpha \int_0^t \int_{\mathbb{R}^4_+} y^b \lvert u^\ast \rvert^2  \overline{\Delta}_b \varphi   \, \mathrm{d}x \,  \mathrm{d}y \, \mathrm{d}\tau \, ,\nonumber
\end{align}
where $b:=1-2\alpha$, $u^\ast$ is the Caffarelli-Silvestre extension of order $\alpha$, the constant $c_\alpha$ the associated normalizing constant and $\overline{\Delta_b} u^\ast:= y^{-b}\div( y^b \overline{\nabla} u^\ast)$.
% If $\alpha \in (1,2)$, we define $b:=3-2\alpha$, $u^\ast$ to be the Caffarelli-Silvestre extension of order $\alpha-1$ and we ask instead of \eqref{eq: localized energy inequality ipo} for 
%\begin{align}\label{eq: localized energy inequality hypo}
%\frac{1}{2} \int_{\mathbb{R}^3}  \lvert u\rvert^2(x,t) \varphi (x,0,t) \, dx +c_{\alpha}&\int_0^t \int_{\mathbb{R}^4_+} y^b\lvert \overline \Delta_b u^* \rvert^2 \varphi \, dx \, dy \, d\tau \nonumber \\ &\leq
%\int_0^t \int_{\mathbb{R}^3} \left( \frac{|u|^2}{2} \partial_t \varphi|_{y=0} + \Big(\frac{| u |^2}{2} +p\Big) u \cdot \nabla \varphi|_{y=0} \right) \, dx \, d \tau\nonumber\\
% &-c_{\alpha} \sum_{i=1}^n \int_0^t\int_{\mathbb{R}^4_+}y^b\overline\Delta_b u^*_i\left(2\overline\nabla\varphi\cdot\overline\nabla u^*_i+u^*_i \overline\Delta_b\varphi \right)\, dx \, dy \, d\tau \, .
%\end{align}
\end{definition}

\subsection{Weak-strong uniqueness}

Leray showed that Leray–Hopf solutions coincide with the classical solutions as long as the latter exist. Indeed, consider two Leray-Hopf solutions $u$ and $v$ with the same initial datum. Assuming smoothness, we can multiply the difference equation by $(u-v)$ and integrate in space. By incompressibility, we obtain
\begin{equation}\label{eq: diffeqLerayHopf}
\frac{1}{2} \frac{\mathrm{d}}{\mathrm{d}t} \int \lvert u-v\rvert^2(x,t) \, \mathrm{d}x + \int \lvert D(u-v) \rvert^2 \, \mathrm{d}x = \int ((u-v) \cdot \nabla)(u-v) \cdot v \, \mathrm{d}x \, .
\end{equation}
Leray noticed that through regularization we can still derive \eqref{eq: diffeqLerayHopf} with an inequality, provided $v\in L^2((0,T), L^\infty)$ only. %The crucial observation is that in the latter case the right-hand side is well-defined; 
Uniqueness then follows from a standard Gr\"onwall argument. As expected, $v\in L^2((0,T), L^\infty)$ still gives uniqueness in the hyperdissipative range $\alpha>1$ (see \cite{CDM}); this requirement can even be weakened using the stronger dissipation on the left-hand side (see forthcoming Proposition \ref{prop: weakstrong uniq}). In the ipodissipative case $0<\alpha<1$ however, it only holds $u-v \in L^2((0,T), H^\alpha)$ and the right-hand side of \eqref{eq: diffeqLerayHopf} is not meaningful, written in this form. Instead assuming smoothness, we observe that by incompressibility
\begin{equation}\label{eq: diffeqLerayHopf ipo}
 \int ((u-v) \cdot \nabla)(u-v) \cdot v \, \mathrm{d}x  =   \int \div \left((u-v) \otimes (u-v)\right) \cdot v \, \mathrm{d}x  \,.
\end{equation}
Integrating by parts $(-\Delta)^{(1-\alpha)/2}$-derivatives on $v$, this allows to deduce a weak-strong uniqueness criterion involving only assumptions on the integrability of $(-\Delta)^{(1-\alpha)/2}v$ and $v$. For $\alpha= 1$, this criterion recovers the classical one $v\in L^2((0,T), L^\infty)$.

\begin{prop}\label{prop: weakstrong uniq} 
Let $0<\alpha<\frac{3}{2}$. Let $(u,p)$ and $(v,q)$ be Leray-Hopf solutions of \eqref{eq: NSalpha} on $\mathbb{R}^3\times (0,T)$ with common initial datum $u_0 \in L^2$ with $\div u_0=0$. If we additionally assume that 
\begin{equation}\label{eq: WS} 
\begin{cases}
 v \in L^2((0,T), L^{\frac 3{\alpha-1}}) \, \qquad& \mbox{if $\alpha \in \big[ 1, \frac 32\big)$}\, ,\\
(-\Delta)^{(1-\alpha)/2} v \in L^2((0,T), L^\infty) \, \qquad& \mbox{if $\alpha \in \big[\frac 34,1\big)$}\, ,\\
(-\Delta)^{(1-\alpha)/2} v \in L^2((0,T), L^\infty) \mbox{ and }  v \in L^2((0,T), L^\frac{3}{\alpha})\, \qquad& \mbox{if $\alpha \in \big(0,\frac 34\big)$}\, ,
\end{cases}
\end{equation}
%and, if $0< \alpha < \frac{3}{4}$, we assume furthermore that $ v \in L^2((0,T), L^\frac{6}{2\alpha})$, 
then $u=v$ in $L^2$ on $(0,T)$.
\end{prop}
%\begin{remark}\label{rmk: weakenhyp on weakstronuniq} Notice that for $\alpha= 1$, we recover the classical criterion $v\in L^2((0,T), L^\infty)$.
%\end{remark}
\begin{proof}
We can assume that $u,v \in C((0,T), L^2_w)$. Using the regularization of Leray and appropriate commutator estimates justifying the integration by parts of $(-\Delta)^{(1-\alpha)/2}$-derivatives on $v$, it is straight-forward to show that for every $t\in (0,T)$ it holds
\begin{align}\label{eq: diffeq leray sol}
\frac{1}{2} \int  \lvert u-v\rvert^2 (x,t)  \, &\mathrm{d}x + \int_0^t \int \lvert (-\Delta)^{\alpha/2} (u-v) \rvert^2(x, \tau) \, \mathrm{d}x \, \mathrm{d}\tau \nonumber  \\ &\leq C \int_0^t \norm{(-\Delta)^{(1-\alpha)/2}v(\cdot, \tau)}_{L^\infty} \int \lvert (-\Delta)^{\alpha/2}(u-v) \rvert  \lvert u-v\rvert (x, \tau)\, \mathrm{d}x \, \mathrm{d}\tau \, .
\end{align}
Observe that under the assumption \eqref{eq: WS}, the right-hand side of \eqref{eq: diffeq leray sol} is well-defined. The additional requirement $v\in L^2((0,T), L^\frac{3}{\alpha})$ in case $0<\alpha < \frac{3}{4}$ comes to ensure that the divergence terms vanish, that is that $\lvert v \rvert^2 \lvert u \rvert, \lvert p \rvert \lvert v \rvert, \lvert u \rvert \lvert q \rvert  \in L^1\,.$ (See proof of Lemma \ref{lem: energyineq for diffeq H1} for a similar reasoning.)
Thus by reabsorbing on the left-hand side, we obtain for every $t\in (0,T)$
\begin{equation*}
\int \lvert u-v \rvert^2 (x,t) \, \mathrm{d}x \leq C  \int_0^t \norm{(-\Delta)^{(1-\alpha)/2}v(\cdot, \tau)}_{L^\infty}^2 \int \lvert u-v \rvert^2 (x, \tau)\, \mathrm{d}x \, \mathrm{d}\tau \, .
\end{equation*}
Since $\tau \mapsto \norm{(-\Delta)^{(1-\alpha)/2}v(\cdot, \tau)}_{L^\infty}^2\in L^1((0,T))$ and $u(\cdot, 0)=v(\cdot, 0)$ in $L^2$, we conclude from Gr\"onwall that $u(\cdot, t)=v(\cdot, t)$ in $L^2$ for any $t \in (0,T)$. For $\alpha > 1$, we conclude analogously, using the right-hand side in the form of \eqref{eq: diffeqLerayHopf} and observing that $\norm{\nabla u}_{L^\frac{6}{5-2\alpha}} \leq C \norm{(-\Delta)^{\alpha/2} u}_{L^2}$.
\end{proof}

\subsection{Stability of constants in Sobolev and commutator estimates}
The stability of all constants in $\alpha$ will be crucial to close the argument, we will, whenever needed, keep track of them rigorously. If not stated otherwise, $C$ will denote a universal constant, independent on $\alpha$, which may change line by line. We denote by $\bar C_{\alpha}$ the constant from the embedding of $\dot{H}^\alpha(\mathbb{R}^3) \hookrightarrow L^\frac{6}{3-2\alpha}(\mathbb{R}^3)$ and we recall from its proof (for instance \cite{Stein}) that those constants are uniformly bounded in $\alpha$ away from the endpoint $\alpha=\frac{3}{2}$.

 Let $0<s<2$, $s_1,s_2 \in [0,1)$, such that $s= s_1+s_2$. Let $p_1, p_2 \in (1,\infty)$ and $r \in [1,\infty)$ be such that $\frac 1 r = \frac 1{p_1}+\frac 1{p_2}$. We recall the following Leibnitz-type inequality from \cite[Theorem A.8]{KPV} and \cite{Dancona}
\begin{equation}
\label{eqn:commuta}
\| (-\Delta)^{s/2}(uv)-u(-\Delta)^{s/2}v- (-\Delta)^{s/2}u v\|_{L^r} \leq C(s, p_1,p_2) \| (-\Delta)^{s_1/2} u \|_{L^{p_1}}\| (-\Delta)^{s_2/2} u \|_{L^{p_2}}.% 
\end{equation}

To track the dependence of our constants explicitly, throughout the paper we call $\bar C$ and $\bar D$ uniform upper bounds for the constant in the Sobolev inequality and in \eqref{eqn:commuta} respectively, that is
$$\bar C:= \sup_{\alpha \in [0,\frac 54]} \max\{1,\bar C_\alpha\}, \qquad
  \bar D(s) := \sup_{p_1,p_2 \in [2, 12]}\max\{1,C(s, p_1,p_2)\}.$$
%which holds for any function $u, v$.
\section{Stability on finite time intervals for $\alpha>0$
%Local-in-time stability for $\alpha>0$} 
%\subsection{Local-in-time stability for $\alpha>0$
}

\begin{prop}\label{prop: locstab H3}
Let $T>0$ arbitrary and $0< \alpha \leq \frac{5	}{4}$. Let $s=3$ if $\alpha \in (0, \frac{5}{4}]$, $s=2$ if $\alpha \in (\frac{1}{2}, \frac{5}{4}]$ and $s\in (\frac{5}{2}-2\alpha,1 ]$ if $\alpha\in (\frac{3}{4}, \frac{5}{4}]$. Suppose that there exists a smooth solution $u \in C([0,T], H^s) \cap L^2([0,T], H^{s+\alpha})$ to the Navier-Stokes equation of order $\alpha$ with divergence-free initial datum $u_0\in H^s(\mathbb{R}^3)$. We additionally assume that 
\begin{equation}
u \in L^2([0,T], H^{s+\alpha+\delta}) \cap L^2([0, T], H^{s+1}) \, .
\end{equation}
for some $\delta \in (0, 1]$. Then there exists $\epsilon>0$ such that for any $v_0\in H^s(\mathbb{R}^3) $ divergence-free and for any $0< \beta \leq \frac{5}{4}$ satisfying 
\begin{equation}\label{eqn:eps-choice}
\norm{u_0-v_0}_{H^s}^2 + \lvert \alpha-\beta \rvert^\delta < \epsilon, 
\end{equation}
there exists a unique solution $v \in C([0,T], H^s) \cap L^2([0,T], H^{s+\beta})$ to the fractional Navier-Stokes equations of order $\beta$ with initial datum $v_0$ which is smooth in the interior.
\end{prop}

%\begin{remark} Notice from Lemma \ref{lem: energyineq for diffeq Hk} and Lemma \ref{lem: energyineq for diffeq H1} that to obtain the stability on finite time intervals in the range $\frac{1}{2}< \alpha < \frac{3}{2}$ $H^2$-regularity of the solution $u$ and the initial datum would be enough, whereas for the range of fractional orders $\frac{3}{4}< \alpha < \frac{3}{2}$ even $H^1$-regularity suffices. %{\color{red} add to statement; comment on 12?}
%\end{remark}

\subsection{An energy inequality for the difference equation} The stability argument employs an estimate of the difference of two fractional Laplacians of order $\alpha$ and $\beta$ respectively in terms of $\lvert \alpha-\beta  \rvert$. This estimate also drives the additional regularity assumption $u_0 \in H^{\delta}$ for $\delta>0$.
\begin{lemma}\label{lem: fraclapRn} Let $s\geq 0$, $\delta \in \mathopen(0,1 \mathclose]$, $\frac{\delta}{2} \leq \beta  \leq \alpha$, and $u \in H^{s+2\alpha+\delta}(\mathbb{R}^n)$. Then we can estimate
\begin{equation} \label{boundeasy fraclapRn}
\norm{[(-\Delta)^{\alpha}-(-\Delta)^{\beta}]u}_{H^s} \leq C (\alpha-\beta)^\delta \norm{u}_{H^{s+2\alpha+\delta}} \,,
\end{equation}
where $C>0$ is a universal constant independent of $s$, $\delta$, $\alpha$ and $\beta$.
\end{lemma}
\begin{proof}
Since the fractional Laplacian commutes with derivatives, it is enough to consider the case $s=0$.  Let $\beta \in [\frac{\delta}{2}, \alpha)$. We write
\begin{equation*}
\norm{[(-\Delta)^{\alpha}-(-\Delta)^{\beta}] u}_{L^2}^2 = \int (\lvert \xi \rvert^{2\alpha}-\lvert \xi \rvert ^{2\beta})^2 \lvert \hat{ u}(\xi) \rvert^2 \, \mathrm{d}\xi = \rm I + II \, ,
\end{equation*}
where we split the integration domain into $\{\lvert\xi\rvert \leq 1\}$ and $\{\lvert \xi\rvert > 1\}$ respectively. Since $1-e^{-x} \leq x \text{ and } 1-\frac{1}{x} \leq \ln x \leq x-1$ for every $x>0$ and since $\left(x_1+x_2\right)^p \leq \max\{2^{p-1}, 1 \} ( x_1^p+x_2^p)$ for $ x_1, x_2>0$ and $p>0$, we have that
%\begin{align*}
%1-e^{-x} &\leq x \text{ and } 1-\frac{1}{x} \leq \ln x \leq x-1 \\
% \left(\sum_{i=1}^n x_i\right)^p &\leq \max\{n^{p-1}, 1 \} \sum_{i=1}^n x_i^p \, .
%\end{align*}to bound
\begin{align*}
\rm I &= \int_{\lvert \xi \rvert \leq 1} (\lvert \xi \rvert^{2\beta}- \lvert \xi \rvert^{2\alpha})^{(2-2\delta)} \lvert \xi \rvert^{4\beta\delta}\left(1- e^{- 2(\alpha-\beta)\lvert \ln \lvert \xi \rvert\rvert }\right)^{2\delta} \lvert \hat{u}(\xi)\rvert^2 \, \mathrm{d}\xi \\
&\leq 2^{2\delta} (\alpha-\beta)^{2\delta} \int_{\lvert \xi \rvert \leq 1} (\lvert \xi \rvert^{2\beta}- \lvert \xi \rvert^{2\alpha})^{(2-{2\delta})} \lvert \xi \rvert^{4\beta \delta} \lvert \ln \lvert \xi \rvert \rvert^{2\delta} \lvert \hat{u}(\xi)\rvert^2 \, \mathrm{d}\xi \\
&\leq 2^{2\delta} (\alpha-\beta)^{2\delta} \max \{2^{1-{2\delta}}, 1\} \int_{\lvert \xi \rvert \leq 1} (\lvert \xi \rvert^{2\beta (2-{2\delta})} +\lvert \xi \rvert^{2\alpha (2-{2\delta})}) \lvert \xi \rvert^{4\beta {\delta}} \lvert \xi \rvert^{-{2\delta}} \lvert \hat{u}(\xi)\rvert^2 \, \mathrm{d}\xi \\
&\leq \max \{2, 2^{2\delta} \} (\alpha-\beta)^{2\delta} \int_{\lvert \xi \rvert \leq 1} (\lvert \xi \rvert^{4\beta-{2\delta}} +\lvert \xi \rvert^{4\alpha-{2\delta}(1+2(\alpha-\beta))}) \lvert \hat{u}(\xi)\rvert^2 \, \mathrm{d}x \, .
\intertext{Similarly, we estimate}
\rm II
&= \int_{\lvert \xi \rvert >1} (\lvert \xi \rvert^{2\alpha}-\lvert \xi \rvert^{2\beta})^{(2-{2\delta})} \lvert \xi \rvert^{4\alpha\delta} \left(1-e^{-2(\alpha-\beta)\ln\lvert\xi\rvert}\right)^{2\delta} \lvert \hat{u}(\xi)\rvert^2 \mathrm{d}\xi \\
&\leq 2^{2\delta} (\alpha-\beta)^{2\delta} \int_{\lvert \xi \rvert >1} (\lvert \xi \rvert^{2\alpha}-\lvert \xi \rvert^{2\beta})^{(2-{2\delta})} \lvert\xi\rvert^{4\alpha \delta} (\ln\lvert \xi\rvert)^{2\delta} \lvert \hat{u}(\xi)\rvert^2\mathrm{d}\xi\\
&\leq 2 (\alpha-\beta)^{2\delta} \int_{\lvert \xi \rvert >1} \lvert\xi\rvert^{4\alpha} \lvert \xi \rvert^{2\delta}  \lvert \hat{u}(\xi)\rvert^2\mathrm{d}\xi \,.
\end{align*}
Collecting terms, we have obtained
\begin{equation*}
\norm{[(-\Delta)^{\alpha}-(-\Delta)^{\beta}]u}_{L^2}^2 \leq (\alpha-\beta)^{2\delta} \left( 2\norm{u}_{\dot{H}^{2\alpha+\delta}} + \max \{2, 2^{2\delta} \}( \norm{u}_{\dot{H}^{2\beta-\delta}}^2 + \norm{u}_{\dot{H}^{2\alpha- {2\delta}(\alpha-\beta)-\delta}}^2) \right) \, . 
\end{equation*}
We conclude by observing that by the interpolation $\norm{u}_{\dot{H}^{2\beta-\delta}}, \norm{u}_{\dot{H}^{2\alpha-{2\delta}(\alpha-\beta)-\delta}} \leq \norm{u}_{H^{2\alpha + \delta}}\,.$
\end{proof}

\begin{lemma}\label{lem: energyineq for diffeq H1} Let $ \frac{3}{4} < \alpha \leq \frac{5}{4} $ and $\frac{5}{2}-2\alpha < s \leq 1$. Let $u_0 \in H^s$ divergence-free. Consider $u \in C(\left[0,T\right], H^s) \cap L^2([0,T], H^{s+\alpha})$, a smooth solution to the Navier-Stokes equation of order $\alpha$ starting from $u_0$. We additionally assume that
\begin{equation}\label{eq: addintdiffeq}
u \in L^2(\left[0, T\right], H^{s+\alpha+\delta}) \text{ and } Du \in L^2([0,T], H^s)  
\end{equation}
for some positive $\delta \in (0, 1]$.  Then, for any $\frac{3}{4} < \beta \leq \frac{5}{4}$ such that 
\begin{equation}\label{eq: hypdiffeqHs}
 \lvert \alpha-\beta \rvert \leq \frac{1}{2} \min  \{ \delta, \frac 12 (s-(\frac{5}{2}-2\alpha)) \}
 \end{equation}
and any smooth solution $v \in C([0,T], H^s) \cap L^2([0,T], H^{s+\beta})$ of the Navier-Stokes equation of order $\beta$, there holds with $f(t)=\norm{(v-u)(t)}_{H^s}^2$ for every $t\in [0,T]$ that
\begin{equation}\label{eq: diffineq}
f(t) \leq f(0) + \int_0^t C_0 f^\gamma(\tau) + C_1 \left( \norm{D u(\tau)}_{H^s}^2 + \norm{u(\tau)}_{H^{s+\alpha+\delta}}^2 \right) f(\tau) \, \mathrm{d}\tau + C_2 \lvert \alpha-\beta \rvert^{\delta} \int_0^t \norm{u(\tau)}_{H^{s+\alpha+\delta}}^2 \, \mathrm{d}\tau \, ,
\end{equation}
where $\gamma=\gamma(s, \beta)=\frac{6\beta-5+2s}{4\beta-5+2s}$, $C_0=C_0(s,\alpha)>0$, $C_1=C_1(s)>0$ and $C_2>0$ are universal. Moreover, the dependence of $C_0$ on $\alpha$ and $s$ is through $
C_0 \leq 2 (12(1+ \bar D) \bar C^2)^{\frac{5}{s-(\frac{5}{2}-2\alpha)}} \, .$
\end{lemma}
\begin{proof}
Set $w:=v-u$, $w_0:= v_0-u_0$ and call $p$ the difference of the pressure terms. Assume for now that $\beta \leq \alpha$ (the other case is handled analogously). By hypothesis, the difference $w \in C(\left[0,T\right], H^s)\cap L^2([0,T], H^{s+\beta})$ is divergence-free and solves the equation
\begin{equation}\label{eq: diffeq}
%\begin{cases}
\partial_t w + (w\cdot \nabla)w+(u\cdot\nabla)w+(w\cdot\nabla)u+ \nabla p=-(-\Delta)^{\beta} w+ \left[(-\Delta)^{\alpha}-(-\Delta)^{\beta}\right]u 
%\\
%\div w  = 0 
%\end{cases}
\end{equation}
with initial datum $w(\cdot, 0)=w_0$.
We multiply the equation by $w\psi_R$ for a cut-off $\psi_R \in C^\infty_c(\mathbb{R}^3)$ such that $\psi_R \equiv 1$ on $B_R(0)$, $0\leq \psi_R \leq 1$ and $\lvert \nabla \psi_R \rvert \leq \frac{C}{R}$. By incompressibility, we obtain
\begin{align*}
\frac{1}{2} \int &\lvert w \rvert^2(x,t) \psi_R \, \mathrm{d}x + \int_0^t \int  (-\Delta)^{\beta} w \cdot w \psi_R  \, \mathrm{d}x \, \mathrm{d}\tau \leq 
\frac{1}{2} \int \lvert w_0 \rvert^2 \psi_R \, \mathrm{d}x + \left \lvert \int_0^t \int (w \cdot \nabla) u \cdot w \psi_R \, \mathrm{d}x \, \mathrm{d}\tau \right \rvert \\ &+ \left \lvert \int_0^t \int [(-\Delta)^\alpha-(-\Delta)^\beta] u \cdot w \, \psi_R \, \mathrm{d}x \, \mathrm{d}\tau \right \rvert + \left \lvert \int_0^t \int \left( w\left( \frac{\lvert w \rvert^2}{2} +p \right) + u \frac{\lvert u \rvert^2}{2} \right) \cdot \nabla \psi_R \, \mathrm{d}x \, \mathrm{d}\tau \right \rvert \, .
\end{align*}
Since $\lvert w \rvert^3 + \lvert w \rvert \lvert p \rvert + \lvert u \rvert^3 \in L^1([0, T], L^1)$ by Sobolev embeddings and Calderon-Zygmund estimates, we deduce that in the limit $R \to \infty$ the third line is negligible.
% Indeed, 
%\begin{equation*}
% \lim_{R \to \infty } \left \lvert \int_0^t \int \left( w\left( \frac{\lvert w \rvert^2}{2} +p \right) + u \frac{\lvert u \rvert^2}{2} \right) \cdot \nabla \psi_R \, \mathrm{d}x \, \mathrm{d}\tau \right \rvert \,  \leq \lim_{R \to \infty} \frac{C}{R} \norm{\lvert w \rvert^3 + \lvert w \rvert \lvert p \rvert + \lvert u \rvert^3}_{L^1([0,T], L^1)} =0 \, .
%\end{equation*}
Passing to the limit $R \to \infty$, we have for $t\in [0,T]$ that
\begin{align*}
\frac{1}{2} \int \lvert w \rvert^2(x,t) \, \mathrm{d}x + \int_0^t \int \lvert (-\Delta)^{\beta/2} w \rvert^2 \, \mathrm{d}x \, \mathrm{d}\tau &\leq 
\frac{1}{2} \int \lvert w_0 \rvert^2 \, \mathrm{d}x + \left \lvert \int_0^t \int (w \cdot \nabla) u \cdot w \, \mathrm{d}x \, \mathrm{d}\tau \right \rvert \\ &+ \left \lvert \int_0^t \int [(-\Delta)^\alpha-(-\Delta)^\beta] u \cdot w \, \mathrm{d}x \, \mathrm{d}\tau \right \rvert \, .
\end{align*}
By Gagliardo-Nirenberg-Sobolev inequality and Young, we estimate
\begin{align*}
\left \lvert \int (w \cdot \nabla) u \cdot w \, \mathrm{d}x  \right \rvert &\leq \bar C \norm{w}_{L^\frac{6}{3-2\beta}} \norm{D u}_{L^\frac{6}{2\beta}} \norm{w}_{L^2} \leq \frac{1}{4} \int \lvert (-\Delta)^{\beta/2} w \rvert^2 \, \mathrm{d}x + \bar C^2 \norm{D u}_{L^\frac{6}{2\beta}}^2 \norm{w}_{L^2}^2 \, .
\end{align*}
To bound the last factor, we observe that as long as $\frac{3}{4} \leq \beta \leq \frac{3}{2}$ it holds by Gagliardo-Nirenberg-Sobolev and interpolation that
\begin{equation}\label{eq: est62beta}
\norm{f}_{L^\frac{6}{2\beta}} \leq \bar C \norm{f}_{\dot{H}^{\frac{3}{2}-\beta}} \leq \bar C  \norm{(-\Delta)^{\beta/2} f}_{L^2}^{\frac{3}{2\beta}-1} \norm{f}_{L^2}^{2-\frac{3}{2\beta}} \, .
\end{equation}
By Plancherel, Young and  Lemma \ref{lem: fraclapRn}, we estimate the dissipative term by
\begin{align}
\left \lvert \int [(-\Delta)^\alpha-(-\Delta)^\beta] u \cdot w \, \mathrm{d}x \right \rvert &\leq \frac{1}{4} \int \lvert (-\Delta)^{\beta/2} w \rvert^2 \, \mathrm{d}x +  \int \lvert [(-\Delta)^{\alpha-\beta/2}-(-\Delta)^{\beta/2}] u \rvert^2 \, \mathrm{d}x  \\\label{eqn:alphabeta}
&\leq \frac{1}{4} \int \lvert (-\Delta)^{\beta/2} w \rvert^2 \, \mathrm{d}x + C_2(\alpha-\beta)^\delta \norm{u}_{H^{\alpha+\delta}}^2 \, ,
\end{align}
where $C_2$ is the universal constant from Lemma \ref{lem: fraclapRn}. Reabsorbing in the left-hand side, we have
\begin{equation}\label{eq: diffineq L2}
\frac{\mathrm{d}}{\mathrm{d}t}\norm{ w(t)}_{L^2}^2+ \norm{(-\Delta)^{\beta/2} w(t)}_{L^2}^2  \leq 2 \bar C^4 \norm{ u(t)}_{\dot{H}^{\frac{5}{2}-\beta}}^2 \norm{w(t)}_{L^2}^2 + 2 C_2 (\alpha-\beta)^\delta \norm{u(t)}_{H^{\alpha+\delta}}^2 \, .
\end{equation}
Consider first $s<1$. Since $(\lvert w  \rvert + \lvert u \rvert )\lvert (-\Delta)^{s/2} w \rvert^2 + \lvert (-\Delta)^{s/2} p \rvert \lvert (-\Delta)^{s/2} w \rvert  \in L^1([0,T], L^1)$, we can argue as before to obtain the following energy inequality for the derivative of order $s$ 
\begin{align*}
\frac{1}{2}&\frac{\mathrm{d}}{\mathrm{d}t} \norm{(-\Delta)^{s/2} w(t)}_{L^2}^2 + \norm{(-\Delta)^{(s+\beta)/2} w (t)}_{L^2}^2 
\\& \leq  \left \lvert \int \left[ 
 (-\Delta)^{s/2}\left[ (w \cdot \nabla) w \right] - (w \cdot \nabla) (-\Delta)^{s/2} w 
  \right] \cdot (-\Delta)^{s/2} w \, \mathrm{d}x \right \rvert \\
&+\left \lvert  \int \left[  (-\Delta)^{s/2} \left[ (u \cdot \nabla) w\right] - (u \cdot \nabla) (-\Delta)^{s/2} w
\right] \cdot (-\Delta)^{s/2} w \, \mathrm{d}x  \right \rvert \\
&+ \left \lvert \int \left[ (-\Delta)^{s/2} \left[ (w \cdot \nabla) u\right] - (w \cdot \nabla) (-\Delta)^{s/2} u + (w \cdot \nabla) (-\Delta)^{s/2} u \right] \cdot (-\Delta)^{s/2} w \, \mathrm{d}x \right \rvert \\
&+ \left \lvert \int \left[(-\Delta)^\alpha-(-\Delta)^\beta \right] (-\Delta)^{s/2} u \cdot (-\Delta)^{s/2} w \, \mathrm{d}x \right \rvert  =: \rm I + II +III+ IV \, .
\end{align*}
%where we have introduced the commutators 
%\begin{align*}
%R_s^1 &:= (-\Delta)^{s/2}\left[ (w \cdot \nabla) w \right] - ((-\Delta)^{s/2} w \cdot \nabla) w - (w \cdot \nabla) (-\Delta)^{s/2} w \, , \\
%R_s^2 &:= (-\Delta)^{s/2} \left[ (u \cdot \nabla) w\right]- ((-\Delta)^{s/2} u \cdot \nabla) w - (u \cdot \nabla) (-\Delta)^{s/2} w \, ,\\
%R_s^3 &:= (-\Delta)^{s/2} \left[ (w \cdot \nabla) u\right]- ((-\Delta)^{s/2} w \cdot \nabla) u - (w \cdot \nabla) (-\Delta)^{s/2} u \, .
%\end{align*}
We estimate line by line. Recall that by interpolation, we can bound as long as $1-\beta \leq s \leq 1$
\begin{equation}\label{eq: estgrad}
\norm{\nabla f}_{L^2} \leq  \norm{(-\Delta)^{(s+\beta)/2} f}_{L^2}^\frac{1-s}{\beta} \norm{(-\Delta)^{s/2}f}_{L^2}^{1-\frac{(1-s)}{\beta}}
\end{equation}
with constant 1. Then the first term is estimated by Gagliardo-Nirenberg-Sobolev, \eqref{eqn:commuta} (with $s_1=s$, $p_1=\frac{6}{2\beta}$, $s_2=0$, $p_2=2$), \eqref{eq: est62beta} and \eqref{eq: estgrad} by
\begin{align*}
 \text{I} &\leq \bar C  \norm{ (-\Delta)^{s/2}\left[ (w \cdot \nabla) w \right] - (w \cdot \nabla) (-\Delta)^{s/2} w 
}_{L^\frac{6}{3+2\beta}} \norm{(-\Delta)^{(s+\beta)/2} w}_{L^2} \\
&\leq \bar C (1+\bar D) \norm{(-\Delta)^{s/2}w}_{L^\frac{6}{2\beta}} \norm{\nabla w}_{L^2}  \norm{(-\Delta)^{(s+\beta)/2} w}_{L^2} \\
 &\leq \bar C^2 (1+\bar D) \norm{(-\Delta)^{(s+\beta)/2}w}_{L^2}^{\frac{5-2s}{2\beta}} \norm{(-\Delta)^{s/2} w}_{L^2}^\frac{6\beta-5+2s}{2\beta} \, .
\end{align*}
The hypothesis  \eqref{eq: hypdiffeqHs} guarantees that $s > \frac{5}{2}-2\beta$, 
and thus we can use Young with exponents $\frac{4\beta}{5-2s}$ and $\frac{4\beta}{4\beta-5+2s}$ to achieve 
\begin{equation*}
\text{I} \leq \frac{1}{12} \norm{(-\Delta)^{(s+\beta)/2}w}_{L^2}^2 + C_0  \norm{(-\Delta)^{s/2}w}_{L^2}^{2 \gamma} \, , 
\end{equation*}
where we introduced $\gamma= \gamma(s, \beta)$ as in the statement and
\begin{align}
%\gamma&= \gamma(s, \beta):= \frac{6\beta-5+2s}{4\beta-5+2s} \, ,  \\
C_0&=C_0(s,\beta):=((1+\bar D) \bar C^2)^\frac{4\beta}{4\beta-5+2s} 
%\left(1-\frac{(5-2s)}{4\beta}\right)
\left( \frac{12(5-2s)}{4\beta}\right)^\frac{5-2s}{4\beta-5+2s}
%> ((1+\bar D) \bar C^2)^\frac{4\beta}{4\beta-5+2s} 
%\left(1-\frac{(5-2s)}{4\beta}\right)
%\left( \frac{12(5-2s)}{4\beta}\right)^\frac{5-2s}{4\beta-5+2s} 
 \leq (12(1+ \bar D) \bar C^2)^\frac{5}{s-(\frac{5}{2}-2\alpha)}
\, . \label{eq: C0Hs}
\end{align}
The second term is estimated similarly (using now \eqref{eqn:commuta} with $s_1=s$, $p_1=\frac{6}{2(s+\beta)-2}$, $s_2=0$, $p_2=\frac{6}{5-2(s+\beta)}$) by
\begin{align*}
\text{II} &\leq (1+\bar D)  \norm{(-\Delta)^{s/2} u}_{L^\frac{6}{2(s+\beta)-2}} \norm{\nabla w}_{L^\frac{6}{5-2(s+\beta)}} \norm{(-\Delta)^{s/2} w}_{L^2} \\ 
&\leq (1+\bar D) \bar C^2 \norm{u}_{\dot{H}^{\frac{5}{2}-\beta}} \norm{(-\Delta)^{(s+\beta)/2}w}_{L^2} \norm{(-\Delta)^{s/2}w }_{L^2}  \\
&\leq \frac{1}{12} \norm{(-\Delta)^{(s+\beta)/2} w}_{L^2}^2 + 3 (1+\bar D)^2 \bar C^4 \norm{u}_{\dot{H}^{\frac{5}{2}-\beta}}^2 \norm{(-\Delta)^{s/2}w}_{L^2}^2 \, .
\end{align*}
We split the third line $\text{III}=\text{III}.1+\text{III}.2$, where $\text{III}.1$ contains the first two addends and $\text{III}.2$ stands for the last one. $\text{III}.1$ is again estimated  using Gagliardo-Nirenberg-Sobolev, \eqref{eqn:commuta}(with $s_1=s$, $p_1=2$, $s_2=0$, $p_2=\frac{6}{2\beta}$) and \eqref{eq: est62beta} by
\begin{align*}
\text{III}.1&\leq  \bar C^2 (1+\bar D) \norm{u}_{\dot{H}^{\frac{5}{2}-\beta}} \norm{(-\Delta)^{s/2}w}_{L^2} \norm{(-\Delta)^{(s+\beta)/2} w}_{L^2} \\
&\leq \frac{1}{24} \norm{(-\Delta)^{(s+\beta)/2} w}_{L^2} ^2 + 6 \bar C^4 (1+\bar D)^2 \norm{u}_{\dot{H}^{\frac{5}{2}-\beta}}^2\ \norm{(-\Delta)^{s/2}w}_{L^2}^2.
\end{align*}
To estimate $\text{III}.2$, we distinguish two cases. Assume first $s+\beta \geq \frac{3}{2}$. Then  $\norm{f}_{\dot{H}^{\frac{3}{2}-\beta}} \leq \norm{f}_{H^s}$ with constant $1$ and thus
\begin{align*}
\text{III}.2 &= \left \lvert \int (w \cdot \nabla) (-\Delta)^{s/2} u \cdot (-\Delta)^{s/2} w \right \rvert \leq \frac{1}{24} \norm{(-\Delta)^{(s+\beta)/2} w}_{L^2} ^2 + 6\bar{C}^4 \norm{u}_{\dot{H}^{1+s}}^2 \norm{w}_{H^s}^2 \, .
\end{align*}
If now $s+\beta<\frac{3}{2}$, we use that $\norm{f}_{L^{\frac{6}{2(s+\beta)}}} \leq \bar C \norm{f}_{\dot{H}^{\frac{3}{2}-(s+\beta)}}$ to obtain
\begin{align*}
\text{III}.2 &\leq \frac{1}{24} \norm{(-\Delta)^{(s+\beta)/2} w}_{L^2} ^2 + 6 \bar C^6 \norm{ u}_{\dot{H}^{\frac{5}{2}-\beta}}^2 \norm{w}_{H^s}^2 \, .
\end{align*}
Finally, the dissipative term is estimated as before by 
\begin{equation*}
\text{IV} \leq \frac{1}{4} \int \lvert (-\Delta)^{(s+\beta)/2} w \rvert^2 \, \mathrm{d}x + C_2(\alpha-\beta)^\delta \norm{(-\Delta)^{s/2}u}_{H^{\alpha+\delta}}^2  \, .
\end{equation*}
Collecting terms, we have obtained after reabsorption of $\frac{1}{2}\int \lvert (-\Delta)^{(s+\beta)/2} w \rvert^2 \, \mathrm{d}x$ on the left
\begin{align*}
\frac{\mathrm{d}}{\mathrm{d}t} \int &\lvert (-\Delta)^{s/2} w \rvert^2 \, \mathrm{d}x \leq 2 C_0 \norm{(-\Delta)^{s/2} w}_{L^2}^{2 \gamma}  + 2 C_2 (\alpha-\beta)^\delta \norm{(-\Delta)^{s/2} u}_{H^{\alpha+\delta}}^2\\ 
&+6 \left[3(1+\bar D)^2 \bar C^4 + 2
%\delta_{s+\beta<\frac{3}{2}} 
\bar C^6 \right] \norm{u}_{\dot{H}^{\frac{5}{2}-\beta}}^2 \norm{w}_{H^s} ^2
+ 12 %\delta_{s+\beta \geq \frac{3}{2}} 
\bar C^4 \norm{u}_{\dot{H}^{1+s}}^2 \norm{w}_{H^s} ^2 \, .
\end{align*}
Under the hypothesis \eqref{eq: hypdiffeqHs}, we can estimate $\norm{u}_{\dot{H}^{\frac{5}{2}-\beta}} \leq \norm{u}_{H^{s+\alpha+\delta}}$ with constant 1
\begin{align}
\frac{\mathrm{d}}{\mathrm{d}t} \norm{ w}_{H^s}^2 + \norm{(-\Delta)^{\beta/2} w}_{H^s}^2  &\leq 2 C_0 \norm{w}_{H^s}^{2 \gamma} + C_1 (\norm{u}_{H^{s+\alpha+\delta}}^2+\norm{Du}_{H^{s}}^2)\norm{w}_{H^s}^2 \nonumber \\
&+ 2 C_2(\alpha-\beta)^\delta \norm{u}_{H^{s+\alpha+\delta}}^2 \, ,
\end{align}
where $C_1:=C_1(s)= 6 \left[3(1+\bar D)^2 \bar C^4 + 2\bar C^6 \right]
$. % is defined through 
%\begin{align}\label{eq: C1Hs}
% C_1( s):= 
% 6 \left[3(1+\bar D)^2 \bar C^4 + 2\bar C^6 \right]
 %\max \bigg  \{ &6\left((1+D_2)^2 \bar C_{s+\beta-\frac{3}{2}}^2 \bar C_{\frac{5}{2}-(s+\beta)}^2 + 2(1+D_3)^2 \bar C_\beta^2 \bar C_{\frac{3}{2}-\beta}^2  + 2 \delta_{s+\beta < \frac{3}{2}} \bar C_{\frac{3}{2}-(s+\beta)}^2 \bar C_\beta^2 \bar C_s^2 \right), \nonumber \\  &\delta_{s+\beta \geq \frac{3}{2}}12 \bar C_\beta^2 \bar C_{\frac{3}{2}-\beta}^2 \bigg \} \, .
%\end{align}
In case $s=1$, the estimates simplify considerably due to the absence of commutators and it is straight-forward to obtain the following energy inequality for the difference equation for $s=1$
\begin{align}\label{eq: diffineq H1}
\frac{\mathrm{d}}{\mathrm{d}t} \norm{w (t)}_{H^1}^2 + \norm{ (-\Delta)^{\beta/2} w(t)}_{H^1}^2  &\leq 2 C_0\norm{ w(t)}_{H^1}^{2\frac{3(2\beta-1)}{4\beta-3} } + 8 \bar C^4 \norm{w(t)}_{H^1}^2 \norm{D u(t)}_{H^1}^2 \nonumber \\
&+ 2C_2 (\alpha-\beta)^\delta  \norm{Du(t)}_{H^{\alpha+ \delta}}^2\, ,
\end{align}
where now
\begin{equation*}\label{eq: C_0}
C_0=C_0(1, \beta):=\left(1-\frac{3}{4\beta} \right) \Big(\frac{6\bar C^\frac{3}{2\beta}}{\beta} \Big)^\frac{4\beta}{4\beta-3}
 \leq (8 \bar C^2)^{\frac{5}{2(\alpha-\frac{3}{4})}}
 \, .\qedhere
\end{equation*}
\end{proof}
\begin{lemma}\label{lem: energyineq for diffeq Hk} Let $ 0 < \alpha < \frac{3}{2} $, $T>0$, and $k\in \mathbb{N}$ such that $k > \frac{5}{2}-2\alpha$. Consider $u \in C(\left[0,T\right], H^k) \cap L^2([0,T], H^{k+\alpha})$, a smooth solution to the fractional Navier-Stokes equations of order $\alpha$ with initial datum $u_0\in H^k$. We additionally assume that
\begin{equation}\label{eq: addintegrab Hk}
u \in L^2(\left[0, T\right], H^{k+\alpha+\delta})\cap L^2([0,T], H^{k+1})
\end{equation}
for some positive $\delta \in (0, \min\{1, 2(k-\frac{5}{2}+2\alpha)\}]$. Then, for any $\frac{\delta}{2} \leq \beta < \frac{3}{2}$ such that 
\begin{equation}\label{eq: smallnessreq Hk}
 \lvert \alpha-\beta \rvert < \frac{\delta}{2}
 \end{equation}
and any solution $v \in C([0,T], H^k) \cap L^2([0,T], H^{k+\beta})$ of the fractional Navier-Stokes equations of order $\beta$, there holds with $f(t)=\norm{(v-u)(t)}_{H^k}^2$ for every $t\in [0,T]$ that
\begin{equation}\label{eq: diff ineq H^k}
f(t) \leq f(0) + C_1 \int_0^t f^2(\tau) + \norm{Du(\tau)}_{H^k}^2   f(\tau) \, \mathrm{d}\tau + C_2 \lvert \alpha-\beta \rvert^{\delta} \int_0^T \norm{u(\tau)}_{H^{k+\alpha+\delta}}^2 \, \mathrm{d}\tau 
\end{equation}
for $C_1=C_1(k)>0$ and a universal $C_2>0$. %Moreover, $C_1(k, \beta)$ depends on $\beta$ only through the constant $\bar C_{ \beta}$ arising from Gagliardo-Nirenberg-Sobolev inequality.
\end{lemma}
\begin{proof}
For $\kappa=0, \dots, k$, we can differentiate the difference equation by $D^\kappa$ (by which we denote any derivative of order $\kappa$) and multiply it by $D^\kappa w$. Since $w$ is incompressible, we obtain
\begin{align*}
\frac{1}{2}\frac{\mathrm{d}}{\mathrm{d}t} \lvert D^\kappa w \rvert^2 &+ \div\left( \frac{\lvert D^\kappa w \rvert^2}{2} w\right) + \sum_{j=1}^\kappa D^j w \star D^{\kappa+1-j} w \cdot D^\kappa w + \div \left( \frac{\lvert D^\kappa w \rvert^2}{2} u\right) \\ 
&+ \sum_{j=1}^{\kappa+1} D^j u \star D^{\kappa+1-j} w \cdot D^\kappa w = - (-\Delta)^\beta D^\kappa w \cdot D^\kappa w + [(-\Delta)^\alpha- (-\Delta)^\beta] D^\kappa u \cdot D^\kappa w \, ,
\end{align*}
where we denote by $\star$ any bilinear expression with constant coefficients. We then test the latter equality with a cut-off $\psi_R$ as in the proof of Lemma \ref{lem: energyineq for diffeq H1} and we note that since $\lvert D^\kappa w \rvert^2 w, \lvert D^\kappa w \rvert^2 u \in L^1([0,T], L^1)$, the contributions of the divergence term vanish in the limit $R\to \infty$. Thus, 
\begin{align*}
\frac{1}{2}& \frac{\mathrm{d}}{\mathrm{d}t} \norm{D^\kappa w(t)}_{L^2}^2 + \norm{(-\Delta)^{\beta/2} D^\kappa w(t)}_{L^2}^2 \leq \sum_{j=1}^\kappa \left \lvert \int D^j w \star D^{k+1-j} w \cdot D^\kappa w \, (x,t) \, \mathrm{d}x \right \rvert \\
&+ \sum_{j=1}^{\kappa +1} \left \lvert \int D^j u \star D^{\kappa+1-j}w  \cdot D^\kappa w \, (x,t) \, \mathrm{d}x \right \rvert + \left \lvert \int [(-\Delta)^\alpha-(-\Delta)^\beta] D^\kappa u \cdot D^\kappa w \, (x,t) \, \mathrm{d}x \right \rvert \, .
\end{align*}
By the Sobolev embeddings, it is straight-forward to check that as long as $k > \frac{5}{2}-2\beta$ it holds
\begin{align*} 
\sum_{j=1}^\kappa \left \lvert \int D^j w \star D^{k+1-j} w \cdot D^\kappa w \, (x,t) \, \mathrm{d}x \right \rvert  &\leq \frac{1}{8}\norm{(-\Delta)^{\beta/2} D^\kappa w(t)}_{L^2}^2 + C \norm{w(t)}_{H^k}^4 \, 
\end{align*}
where $C=C(k, \beta)$% depends on $\beta$ through the constant $\bar C_{ \beta}$ arising from Gagliardo-Nirenberg-Sobolev embedding
. Notice that $k> \frac{5}{2}-2\beta$ is guaranteed through  \eqref{eq: smallnessreq Hk} and the upper bound on $\delta$.
As for the second line, we have
\begin{align*} 
\sum_{j=1}^{\kappa +1}\int \lvert D^j u \rvert \lvert D^{\kappa+1-j} w \rvert \lvert D^\kappa w \rvert (x,t) \, \mathrm{d}x 
&\leq \bar C \sum_{j=1}^{\kappa+1} \norm{(-\Delta)^{\beta/2} D^\kappa w (t)}_{L^2} \norm{\lvert D^j u \rvert \lvert D^{\kappa+1-j}w\rvert(t)}_{L^\frac{6}{3+2\beta}} \\
&\leq \frac{1}{8} \norm{(-\Delta)^{\beta/2}D^\kappa w(t)}_{L^2}^2 + C \norm{w(t)}_{H^k}^2 \norm{Du(t)}_{H^k}^2 \, .
\end{align*}
We now estimate the final term as before, using Plancherel, H\"older,  Young and Lemma \ref{lem: fraclapRn} as in \eqref{eqn:alphabeta}.
%\begin{align*}
%\int \left[(-\Delta)^{\alpha}-(-\Delta)^{\beta}\right] D^\kappa u \cdot D^\kappa &w (x,t)  \, \mathrm{d}x
%&\leq \frac{1}{4}  \norm{(-\Delta)^{\beta/2} D^\kappa w(t)}_{L^2}^2  + C_2 (\alpha-\beta)^\delta \norm{D^\kappa u(t)}_{H^{\alpha+ \delta}}^2 \, ,
%\end{align*}
%where $C_2$ is the universal constant of Lemma \ref{lem: fraclapRn}. 
Combining the estimates, reabsorbing the contributions of $\norm{(-\Delta)^{\beta/2} D^\kappa w }_{L^2}^2$ on the left and summing over $\kappa=0, \dots, k$, we obtain
\begin{align*}
\frac{\mathrm{d}}{\mathrm{d}t} \norm{w(t)}_{H^k}^2 + \norm{(-\Delta)^{\beta/2}w(t)}_{H^k}^2 &\leq C_1( \norm{w(t)}_{H^k}^4+ \norm{u(t)}_{H^{k+1}}^2\norm{w(t)}_{H^k}^2) + C_2(\alpha-\beta)^\delta \norm{u(t)}_{H^{k+\alpha+\delta}}^2 \, .
\end{align*}
with $C_1=C_1(k)>0$ and a universal $C_2>0$.
\end{proof}
{
\begin{proof}[Proof of Proposition~\ref{prop: locstab H3}] Consider first the case $\alpha \in (\frac{3}{4}, \frac{5}{4}]$. Let $s\in (\frac{5}{2}-2\alpha,1]$, $v_0 \in H^s$ divergence-free and $\beta \in (0, \frac{5}{4}]$ such that \eqref{eq: addintdiffeq} holds for an $\epsilon \in (0,1]$, yet to be determined. By choosing $\epsilon$ suitably small, we can always assume that $\lvert \alpha-\beta \rvert \leq \frac{1}{2} \min \{\delta, s-(\frac{5}{2}-2\alpha) \}$. Observe that then in particular $s>\frac{5}{2}-2\beta$, where the latter is the critical Sobolev regularity with respect to the natural scaling of \eqref{eq: NSalpha}. Classical arguments allow to build a maximal local solution $v\in C([0,T_{max}), H^s) \cap L^2([0,T_{max}), H^{s+\beta})$ to the Navier-Stokes euqations of order $\beta$ starting from $v_0$ which is smooth in the interior and unique among Leray-Hopf solutions by Proposition \ref{prop: weakstrong uniq}. Moreover, if $T_{max}< + \infty$ we have $\limsup_{t \uparrow T^\ast} \norm{v(t)}_{H^s}=+\infty$.
%Classical arguments allow to build a maximal local solution $v\in C([0,T_{max}], H^s) \cap L^2([0,T_{max}], H^{s+\beta})$ to the Navier-Stokes euqations of order $\beta$ starting from $v_0$ with time of existence $T_{max}\geq T(\norm{v_0}_{H^s})$ which is smooth in the interior and unique among Leray-Hopf solutions by Proposition \ref{prop: weakstrong uniq}. 
%Uniqueness allows to continue this local solution until a maximal time of existence $T_{max} \in \mathbb{R}_+\cup \{ + \infty \}$ such that if $T_{max}< + \infty$ we have $\limsup_{t \uparrow T^\ast} \norm{v(t)}_{H^s}=+\infty$. We want to show that $T_{max}> T$. 
Let us define $f(t):= \norm{(u-v)(t)}_{H^s}$ for $t \in [0, \min \{T_{max}, T \})$. By Lemma \ref{lem: energyineq for diffeq H1}, $f$ satisfies the differential inequality \eqref{eq: diffineq} and hence, %. A standard exercise in calculus gives that 
for any $t\in [0, \min \{T_{max}, T \})$, we have an upper abound on $\max_{s\in [0,t]} f(s) \,.$
%\begin{align*}
%&\max_{s\in [0,t]} f(s) \\ &\leq \frac{(f(0)+C_2 \lvert \alpha-\beta\rvert^\delta \norm{u}_{L^2([0, T], H^{s+\alpha+\delta})}^2) e^{C_1\int_0^t \norm{Du(\tau)}_{H^s}^2 + \norm{u(\tau)}_{H^{s+\alpha+\delta}}^2 \, \mathrm{d}\tau}}{\left(1-(f(0)+C_2 \lvert \alpha-\beta\rvert^\delta \norm{u}_{L^2([0, T], H^{s+\alpha+\delta})}^2)^{\gamma-1} C_0 (\gamma-1) t e^{C_1(\gamma-1)\int_0^t \norm{Du(\tau)}_{H^s}^2 +\norm{u(\tau)}_{H^{s+\alpha+\delta}}^2 \, \mathrm{d}\tau} \right)^\frac{1}{\gamma-1} } \, ,
%\end{align*}
%where we recall from Lemma \ref{lem: energyineq for diffeq H1} that $\gamma-1=\frac{2\beta}{4\beta-5+2s}$.
In particular, if
\begin{equation}\label{eq: localsmallnesreq}
(f(0)+C_2 \lvert \alpha-\beta\rvert^\delta \norm{u}_{L^2([0, T], H^{s+\alpha+\delta})}^2)^{\gamma-1} C_0 (\gamma-1) e^{C_1(\gamma-1)(\norm{u}_{L^2([0,T],H^{s+1})}^2 +\norm{u}_{L^2([0,T],H^{s+\alpha+\delta})}^2)}< \frac{1}{2T} \, ,
\end{equation}
where we recall from Lemma \ref{lem: energyineq for diffeq H1} that $\gamma-1=\frac{2\beta}{4\beta-5+2s}$, we have $$\max_{s\in [0,T]} f(s) \leq ((\gamma-1)  C_0 T)^{-\frac{1}{\gamma-1}} \, .$$ We deduce that $\limsup_{t \uparrow T} \norm{(u-v)(t)}_{H^s} <+\infty$ and thus $T_{max}> T$.  The condition \eqref{eq: localsmallnesreq} is thus satisfied, if we require that \eqref{eqn:eps-choice} is enforced with
\begin{align}\label{eq: localchoiceofepsilon}
\epsilon:= \min \{ &\frac{\delta}{2}, \frac{1}{4}(s-(\frac{5}{2}-2\alpha)), \\ 
&\max\{1, C_2 \norm{u}_{L^2([0,T], H^{s+\alpha+\delta})}^2 \} ( C_0 (\gamma-1)T )^{-\frac{1}{\gamma-1}} e^{-C_1 (\norm{u}_{L^2([0,T],H^{s+1})}^2+\norm{u}_{L^2([0,T],H^{s+\alpha+\delta})}^2)}\big \} \, .
\end{align}
%Recall that a priori the constants $C_0$ and $C_1$ depend on $\beta$. By the explicit form of the constants from Lemma \ref{lem: energyineq for diffeq H1} and the stability of the Sobolev embeddings, we may bound them uniformly in $\beta$. 
Recall that $\gamma$ depends on $\beta$; however, by choosing $\beta$ close enough to $\alpha$, we can bound $\gamma-1$ uniformly away from $0$. This concludes the proof. The cases $\alpha \in (\frac{1}{2}, \frac{5}{4}]$ with $s=2$ and $\alpha \in (0, \frac{5}{4}]$ with $s=3$ follow analogously from Lemma \ref{lem: energyineq for diffeq Hk}. In the latter case, the local existence from $H^3$ initial data follows for instance from \cite[Theorem 3.4]{BM} (the proof there covers the classical Navier-Stokes 
$\alpha=1$ and the Euler equations, and, being based on energy methods, can easily be adapted to the fractional Navier-Stokes equations of order $\alpha>0$).
\end{proof}
}

\section{Leray's estimate on singular times}\label{sec:leray}
In his seminal paper \cite{Leray}, Leray showed that if $u_0 \in H^1$, then the Leray-Hopf solution is unique and smooth for a short time with upper bound $T= C \norm{\nabla u_0}_{L^2}^{-4}$. Thanks to energy inequality, this bound can be iterated to get global existence provided that a Leray-Hopf solution exists and is smooth until a sufficiently large $T^\ast$. With minor modifications, Leray's argument applies to the fractional setting. Notice however, that for an eventual regularization of Leray-Hopf solutions in the ipodissipative range $\alpha \leq 1$, we need an upper bound of the form $T= C \norm{(-\Delta)^{\alpha/2} u_0}_{L^2}^{-\beta}$, for some $\beta>0$, since the energy inequality now only controls $(-\Delta)^{\alpha/2}u$ in $L^2(\mathbb{R}^3 \times [0, +\infty))$. Since $\norm{(-\Delta)^{\alpha/2} u_0 }_{L^2}$ is critical with respect to the natural scaling of \eqref{eq: NSalpha} at $\alpha=\frac{5}{6}$, such an estimate can only be expected in the subcritical range $\alpha> \frac{5}{6}$.

\begin{prop}\label{prop: LeraySingTime} Let $\frac{3}{4} < \alpha \leq \frac{5}{4}$ and $u_0 \in H^1(\mathbb{R}^3)$ divergence-free. Then there exists a universal $C_2=C_2(\alpha)$, uniformly bounded away from $\alpha=\frac{3}{4}$, such that, setting 
\begin{equation} \label{eq: LeraySingTimeDef}
T(\norm{\nabla u_0}_{L^2}, \alpha):= C_2 \norm{\nabla u_0}_{L^2}^{-\frac{4\alpha}{4\alpha-3}} \, ,
\end{equation}
there exists a unique solution $u$ to \eqref{eq: NSalpha} on $(0,T)$ satisfying $u \in L^\infty([0,T), H^1) \cap L^2([0,T),H^{1+\alpha})$ which is smooth in $(0,T)$.
\end{prop}

\begin{proof}
Notice first that by Sobolev embeddings, we have $(-\Delta)^{(1-\alpha)/2} u \in L^2([0,T), H^{2\alpha}) \hookrightarrow L^2([0,T), L^\infty)$ for $\frac{3}{4}< \alpha<1$, whereas for $\alpha \geq 1$ it holds $u\in L^2([0,T), L^\frac{3}{\alpha-1})$. Thus the uniqueness of the Leray-Hopf solution on $[0,T)$ follows from Proposition \ref{prop: weakstrong uniq}. The smoothness in the interior follows from a standard boot-strap argument. We therefore focus on the Leray-Hopf solution $u$ which is attained as limit of the approximation scheme \eqref{eq: regularizedNS}. We perform all the estimates on the unique, smooth and global solutions $(u_\epsilon, p_\epsilon)$ of \eqref{eq: regularizedNS} and pass to the limit  $\epsilon \to 0$ only at the very end. 
By smoothness we may derive the equation by $\partial_j$ and multiply it by $\partial_j u_\epsilon$. To make the computation rigorous we employ a cutoff $\psi_R \in C^\infty_c(B_{2R}(0))$ and we then let $R\to \infty$; the pressure term can be neglected by Calderon-Zygmund estimates, $\norm{\partial_j(u_\epsilon \ast \varphi_\epsilon)}_{L^2_t L^\infty_x} \leq C \norm{\partial_j u_\epsilon}_{L^2_{t,x}}$ and $\norm{u_\epsilon \ast \varphi}_{L^\infty} \leq C \norm{u_\epsilon}_{L^\infty_t L^2_x}$, which give $\partial_j p_\epsilon \in L^2_{t, x}$.  
We obtain for $t\in [0,+\infty)$
\begin{equation}\label{eq: energyineq H1}
\frac{1}{2} \int \lvert Du_\epsilon \rvert^2 (x,t) \, \mathrm{d}x + \int_0^t \int \lvert (-\Delta)^{\alpha/2} Du_\epsilon \rvert^2 \, \mathrm{d}x \, \mathrm{d}\tau \leq \frac{1}{2} \int \lvert Du_0 \rvert^2 \, \mathrm{d}x+ \int_0^t \int \lvert D(u_\epsilon \ast \varphi_\epsilon) \rvert \lvert Du_\epsilon \rvert^2 \, \mathrm{d}x \, \mathrm{d}\tau \, .
\end{equation}
We estimate the right-hand side for $\frac{3}{4}\leq \alpha < \frac{3}{2}$ using H\"older, Young's convolution inequality and the Gagliardo-Nirenberg-Sobolev inequality by
\begin{align}
\int \lvert D(u_\epsilon \ast \varphi_\epsilon) \rvert \lvert Du_\epsilon \rvert^2 \, \mathrm{d}x  &\leq \norm{Du_\epsilon \ast \varphi_\epsilon}_{L^2} \norm{Du_\epsilon}_{L^4}^2 \leq \norm{Du_\epsilon}_{L^2}  \norm{Du_\epsilon}_{L^\frac{6}{3-2\alpha}}^{2\theta} \norm{Du_\epsilon}_{L^2}^{2(1-\theta)}  \nonumber \\
&\leq \bar C^{2\theta} \norm{(-\Delta)^{\alpha/2} Du_\epsilon}_{L^2}^{2\theta} \norm{Du_\epsilon}_{L^2}^{3-2\theta} \label{eq: estonL3norm}\, .
\end{align}
where $\theta=\theta(\alpha):= \frac{3}{4\alpha}$ solves $
\frac{1}{4} = \frac{\theta(3-2\alpha)}{6} + \frac{1-\theta}{2} \, .$
Hence, for $\alpha> \frac{3}{4}$, we may apply Young with exponents $\frac{1}{\theta} = \frac{4\alpha}{3}$ and $ \frac{4\alpha}{4\alpha-3}$ to obtain
\begin{align*}
\int \lvert Du_\epsilon \ast \varphi_\epsilon \rvert \lvert Du_\epsilon \rvert^2 \, \mathrm{d}x \leq \frac{3}{4\alpha}\norm{(-\Delta)^{\alpha/2} Du_\epsilon}_{L^2}^2 + \frac{(4\alpha-3) \bar C^{\frac{8\alpha \theta}{4\alpha-3}}}{4\alpha} \norm{Du_\epsilon}_{L^2}^{\frac{4\alpha(3-2\theta)}{4\alpha-3}} \, .
\end{align*}
Reabsorbing $\frac{3}{4\alpha}\norm{(-\Delta)^{\alpha/2} Du_\epsilon}_{L^2}^2$ on the left-hand side of \eqref{eq: energyineq H1} yields
\begin{equation}\label{eq: energyineq H1 final}
\frac{\mathrm{d}}{\mathrm{d}t} \int \lvert Du_\epsilon \rvert^2 \, \mathrm{d}x  +\left(1-\frac{3}{4\alpha}\right) \int \lvert (-\Delta)^{\alpha/2} Du_\epsilon \rvert^2 \, \mathrm{d}x \leq \frac{2(4\alpha-3)C_2^{-1}}{4\alpha} \left( \int \lvert Du_\epsilon \rvert^2 \, \mathrm{d}x \right)^\beta \, ,
\end{equation}
where $\beta=\beta(\alpha):= \frac{3(2\alpha-1)}{4\alpha-3}$ and $C_2=C_2(\alpha)=\bar C^\frac{-6}{4\alpha-3}$. Setting 
\begin{equation*}
T=T(\norm{\nabla u_0}_{L^2}, \alpha) := \frac{4\alpha}{2(4\alpha-3) \bar C^\frac{6}{4\alpha-3} (\beta-1) \norm{\nabla u_0}_{L^2}^{2(\beta-1)}} =  C_2\norm{\nabla u_0}_{L^2}^{-\frac{4\alpha}{4\alpha-3}} \, ,
\end{equation*}
we have that for any $0 \leq t < T$ the estimate
\begin{equation}\label{eq: uniform bound on H1}
\norm{Du_\epsilon(t)}_{L^2}^2 \leq \frac{\norm{\nabla u_0}_{L^2}^2}{(1-C_2^{-1}t \norm{\nabla u_0}_{L^2}^{2(\beta-1)})^\frac{1}{\beta-1}} \, .
\end{equation}
Recalling \eqref{eq: energyineq H1 final}, we infer that  $\{u_\epsilon\}_{\epsilon >0}$ is uniformly bounded in $L^\infty([0,T), H^1)\cap L^2([0,T), H^{1+\alpha})$. By a standard argument, we can now pass to the limit $\epsilon \to 0$ using weak lower semicontinuity and the strong convergence of $u_\epsilon \rightarrow u$ in $L^p_{loc}(\mathbb{R}^3\times [0,T])$ for $p<\frac{6+4\alpha}{3}$.
%Up to subsequence, we have by Banach-Anaoglu and the strong convergence $u_\epsilon \rightarrow u$ in $L^p_{loc}(\mathbb{R}^3\times [0,T])$ for $p<\frac{6+4\alpha}{3}$ that $Du_\epsilon \rightharpoonup Du$ and $(-\Delta)^{\alpha/2}Du_\epsilon \rightharpoonup (-\Delta)^{\alpha/2} Du$ converge weakly in $L^2(\mathbb{R}^3 \times [0,T))$. By weak lower semicontinuity, we immediately infer
%\begin{equation*}
%\norm{(-\Delta)^{\alpha/2}Du}_{L^2(\mathbb{R}^3 \times 0,T))} \leq \liminf_{\epsilon \to 0} \norm{(-\Delta)^{\alpha/2}Du_\epsilon}_{L^2(\mathbb{R}^3 \times 0,T))} \,
%\end{equation*}
%and that for any $0\leq \sigma<\tau\leq t <T $
%\begin{equation*}
%\int_\sigma^\tau \int \lvert Du \rvert^2 (x, s) \, \mathrm{d}x \, \mathrm{d}s\leq \liminf_{\epsilon \to 0} \int_\sigma^\tau \int \lvert Du_\epsilon \rvert^2 (x, s) \, \mathrm{d}x \, \mathrm{d}s \leq (\tau-\sigma) \frac{\norm{\nabla u_0}_{L^2}}{(1-C_2^{-1} t \norm{\nabla u_0}_{L^2}^{2(\beta-1)})^\frac{1}{\beta-1}} \, ,
%\end{equation*}
%hence $Du \in L^\infty([0,T), L^2)$ with the uniform bound \eqref{eq: uniform bound on H1}.
\end{proof}

\begin{prop}\label{prop: LeraySingTimeHalpha} Let $\frac{5}{6} < \alpha < 1$ and $u_0 \in W^{\alpha,2}(\mathbb{R}^3)$ divergence-free. Then there exists a universal $C_1=C_1(\alpha)$, uniformly bounded away from $\alpha=\frac{5}{6}$, such that, setting 
\begin{equation} \label{eq: LeraySingTimeDefHalpha}
T(\norm{(-\Delta)^{\alpha/2} u_0}_{L^2}, \alpha):= C_1 \norm{(-\Delta)^{\alpha/2} u_0}_{L^2}^{-\frac{4\alpha}{6\alpha-5}} \, ,
\end{equation}
there exists a unique solution $u$ to \eqref{eq: NSalpha} on $(0,T)$ satisfying $u \in L^\infty([0,T), W^{\alpha,2}) \cap L^2([0,T), W^{2\alpha, 2})$ which is smooth on $(0,T)$.
\end{prop}
\begin{proof}
We argue as in the proof of Proposition \ref{prop: LeraySingTime} to obtain the energy inequality for the regularized system, use commutator estimates as in the proof of Lemma \ref{lem: energyineq for diffeq H1}, Gagliardo-Nirenberg-Sobolev and interpolation to obtain
\begin{align*}
\frac{1}{2}\frac{\mathrm{d}}{\mathrm{d}t} \int \lvert(-\Delta)^{\alpha/2} u_\epsilon\rvert^2 \, \mathrm{d}x &+  \int \lvert (-\Delta)^{\alpha} u_\epsilon \rvert^2 \, \mathrm{d}x \leq \left  \lvert \int (-\Delta)^{\alpha/2} \left( (u_\epsilon \ast \varphi_\epsilon) \cdot \nabla) u_\epsilon\right) \cdot (-\Delta)^{\alpha/2} u_\epsilon \, \mathrm{d}x \, \right \rvert \\
&\leq (1+ \bar D) \bar C \norm{(-\Delta)^{\alpha} u_\epsilon}_{L^2} \norm{(-\Delta)^{\alpha/2} ( u_\epsilon \ast \varphi_\epsilon)}_{L^2} \norm{(-\Delta)^{\alpha/2} u_\epsilon}_{L^ \frac{6}{4\alpha-2}}  \\
&\leq (1+ \bar D)\bar C ^\frac{5-2\alpha}{2\alpha} \norm{(-\Delta)^\alpha u_\epsilon}_{L^2}^\frac{5-2\alpha}{2\alpha} \norm{(-\Delta)^{\alpha/2} u_\epsilon}_{L^2}^\frac{8\alpha-5}{2\alpha} \, ,
\end{align*}
by \eqref{eqn:commuta} applied with $s_1=\alpha$, $p_1=2$, $s_2=0$ and $p_2=\frac{6}{5-4\alpha}$.
For $\alpha>\frac{5}{6}$, we use Young to obtain after reabsorption
\begin{equation}\label{eq: Halpha ipo}
\frac{\mathrm{d}}{\mathrm{d}t} \int \lvert (-\Delta)^{\alpha/2} u_\epsilon \rvert^2(x,t) \, \mathrm{d}x \leq \frac{2 (6\alpha-5)C_1^{-1}}{4\alpha} \left( \int \lvert (-\Delta)^{\alpha/2} u_\epsilon \rvert^2(x,t) \, \mathrm{d}x \right)^\frac{8\alpha-5}{6\alpha-5} \, ,
\end{equation}
where $C_1=C_1(\alpha):=(1+\bar  D)^\frac{-4\alpha}{6\alpha-5} \bar C^\frac{-10+4\alpha}{6\alpha-5}\, .$
%\begin{equation}\label{eq: Calpha Halpha ipo}
%C_1=C_1(n, \alpha):=((1+D)\bar C_{ 2\alpha-1})^\frac{4\alpha}{6\alpha-5} \bar C_{\alpha}^\frac{2(5-4\alpha)}{6\alpha-5}\,.
%\end{equation} 
Defining $T$ by \eqref{eq: LeraySingTimeDefHalpha}, we conclude as before.
\end{proof}

\begin{corollary}[Leray's Estimate on Singular Times] \label{cor: LeraySingTime} Let $\frac{5}{6} < \alpha < \frac{5}{4}$ and $u_0 \in L^2(\mathbb{R}^3)$ divergence-free with $\norm{u_0}_{L^2} \leq M$. Then there exists $T^\ast = T^\ast (M, \alpha)>0$ and a Leray-Hopf solution $u$ to \eqref{eq: NSalpha} which is smooth on $[T^\ast, + \infty)$. Moreover, $T^\ast$ is uniformly bounded for $ \alpha \in (\frac{5}{6}, \frac{5}{4})$ and 
\begin{equation}\label{eq: Tstar behavat54}
\lim_{\alpha \uparrow \frac{5}{4}} T^\ast(M, \alpha) =0 \, .
\end{equation}
\end{corollary}
\begin{proof} Let $T^\ast>0$. From the energy inequality, we infer that there exists $\bar{t} \in (0,T^\ast)$ such that 
\begin{equation*}
\norm{(-\Delta)^{\alpha/2} u(\cdot, \bar{t})}_{L^2}^2 \leq \frac{\norm{u_0}_{L^2}^2}{2T^\ast}  \leq \frac{M^2}{2T^\ast} \, .
\end{equation*}
Consider first $\alpha<1$. By Proposition \ref{prop: LeraySingTimeHalpha} the Leray-Hopf solution with initial datum $u(\cdot, \bar{t})$ is smooth and unique until
\begin{equation*}
{C_1}{\norm{(-\Delta)^{\alpha/2} u(\cdot, \bar{t})}_{L^2}^\frac{-4\alpha}{6\alpha-5}} \geq {C_1(2T^\ast)^\frac{2\alpha}{6\alpha-5}}{M^\frac{-4\alpha}{6\alpha-5}} \, .%\frac{C_1}{\norm{(-\Delta)^{\alpha/2} u(\cdot, \bar{t})}_{L^2}^\frac{4\alpha}{6\alpha-5}} \geq \frac{C_1(2T^\ast)^\frac{2\alpha}{6\alpha-5}}{M^\frac{4\alpha}{6\alpha-5}} \, .
\end{equation*}
If we choose $T^\ast= T^\ast(M, \alpha)$ large enough, such that 
\begin{equation}\label{eq: choiceofTstar ipo}
{C_1(2T^\ast)^\frac{2\alpha}{6\alpha-5}}{M^\frac{-4\alpha}{6\alpha-5}} > T^\ast  \text{, or equivalently, } T^\ast > {M^\frac{4\alpha}{5-4\alpha}}{2^\frac{-2\alpha}{5-4\alpha}C_1^\frac{-6\alpha-5}{5-4\alpha}}\, ,
\end{equation}
then the Leray-Hopf solution is smooth and unique on $(\bar{t}, \bar{t}+T^\ast)$ and we can iterate this procedure thanks to the energy inequality.
If $\alpha\geq1$, we notice that by interpolation of Sobolev spaces 
\begin{equation*}
\norm{\nabla u(\cdot, \bar{t})}_{L^2} \leq \norm{u(\cdot, \bar{t})}_{L^2}^{1-\frac{1}{\alpha}} \norm{(-\Delta)^{\alpha/2}u(\cdot, \bar{t})}_{L^2}^\frac{1}{\alpha} \leq M(2T^\ast)^{-\frac{1}{2 \alpha}} \, ,
\end{equation*}
and hence from Proposition \ref{prop: LeraySingTime} the Leray-Hopf solution starting at $\bar{t}$ is smooth and unique until 
\begin{equation*}
C_2\norm{\nabla u(\cdot, \bar{t})}_{L^2}^{-\frac{4\alpha}{4\alpha-3}} \geq C_2(2T^\ast)^\frac{2}{4\alpha-3}M^{-\frac{4\alpha}{4\alpha-3}} \, .
\end{equation*}
As in the ipodissipative case, choosing $T^\ast=T^\ast(M, \alpha)$ large enough satisfying 
\begin{equation}\label{eq: choiceofTstar hyp}
C_2(2T^\ast)^\frac{2}{4\alpha-3}M^{-\frac{4\alpha}{4\alpha-3}} > T^\ast \text{, or equivalently, } T^\ast > M^\frac{4\alpha}{5-4\alpha}2^{-\frac{2}{5-4\alpha}} C_2^{-\frac{4\alpha-3}{5-4\alpha}}
\end{equation} 
allows to iterate the argument and build a global solution starting at $\bar{t}$. Observe now that the fact that $T^\ast(M, \alpha)$ is uniformly bounded away from $\alpha=\frac{5}{4}$ follows from its choice in  \eqref{eq: choiceofTstar ipo} and \eqref{eq: choiceofTstar hyp} and the explicit expression of the constants $C_1$ and $C_2$. Let us now establish \eqref{eq: Tstar behavat54}. Take a non-regular time $T>0$ of a Leray-Hopf solution $(u, p)$ for the Navier-Stokes equations of order $\alpha\geq 1$. For almost every  $0<t<T$, $u(\cdot, t) \in W^{\alpha,2}$ and thus by Proposition \ref{prop: LeraySingTime}, we must have
\begin{equation*}
T-t > C_2\norm{\nabla u(\cdot, t)}_{L^2}^{-\frac{4\alpha}{4\alpha-3}} \geq C_2\norm{(-\Delta)^{\alpha/2} u(\cdot, t)}_{L^2}^{-\frac{4}{4\alpha-3}} \norm{u_0}_{L^2}^{-\frac{4(\alpha-1)}{4\alpha-3}} \, .
\end{equation*}
The last inequality follows by interpolation of Sobolev spaces.
%, we infer
%\begin{equation*}
%C_2\norm{(-\Delta)^{\alpha/2} u(\cdot, t)}_{L^2}^{-\frac{4}{4\alpha-3}} \norm{u_0}_{L^2}^{-\frac{4(\alpha-1)}{4\alpha-3}} < T-t \,.
%\end{equation*}
Recalling $C_2 = \bar C^{-\frac{6}{4\alpha-3}}$ and using the energy inequality, we deduce that
\begin{equation*}
\frac{1}{2} \norm{u_0}_{L^2}^2 \geq \int_0^T \norm{(-\Delta)^{\alpha/2} u(t)}_{L^2}^2 \, \mathrm{d}t \geq \frac{C_2^\frac{4\alpha-3}{2}}{\norm{u_0}_{L^2}^{2\alpha-2}} \int_0^T \frac{\mathrm{d}t}{(T-t)^{1-\frac{5-4\alpha}{2}}} =  \frac{2 T^\frac{5-4\alpha}{2}}{\bar C^3\norm{u_0}_{L^2}^{2\alpha-2}(5-4\alpha)}  \, .
\end{equation*}
Hence we have found the following improved upper bound for $T^\ast$ for $\alpha \geq 1$
\begin{equation}\label{eq: upperbound for Tstar hyp}
T^\ast(M, \alpha) \leq \left(\frac{(5-4\alpha)}{2}\bar C^3 M^{2\alpha} \right)^\frac{2}{5-4\alpha} \,.
\end{equation}
%By the stability of the global embedding $W^{\alpha,2} \hookrightarrow L^\frac{6}{3-2\alpha}$, 
%Since we may bound $\bar C_{\alpha}$ uniformly  in $\alpha \in[ 1,\frac54]$, w
We deduce that for $M>0$ fixed
\begin{equation*}
\lim_{\alpha \uparrow \frac{5}{4}} T^\ast(M, \alpha) \leq \lim_{\alpha \uparrow \frac{5}{4}}(\bar C^3 M^{2\alpha} (5-4\alpha))^\frac{2}{5-4\alpha} = 0 \, .\qedhere
\end{equation*}
\end{proof}

\section{Eventual regularization of suitable weak solutions}\label{sec:reg-suit}
%Let us introduce the parabolic space-time cylinders which are adapted to the natural scaling of \eqref{eq: NSalpha}: For $(x_0, t_0) \in \mathbb{R}^3 \times [0,+ \infty)$ we set $Q_r(x_0, t_0) := B_r(x_0) \times (t_0-r^{2\alpha}, t_0]$. 
The eventual regularization property for $\frac{3}{4}<\alpha\leq\frac56$ can be obtained relying on partial regularity results.
The difference with Leray's estimate of Section~\ref{sec:leray} on the eventual regularization time $T^\ast$ consist in the fact that the dependence of $T^\ast$ on $M$ and $\alpha$ cannot be made explicit.% (and it is applicable to the - a priori - smaller class of suitable weak solutions only).
\begin{prop}\label{prop: eventualreg} Let $\frac{3}{4}<\alpha\leq1$ and $p\in [1,2)$. Consider $u_0 \in L^p \cap L^2$ divergence-free with $\norm{u_0}_{L^2 \cap L^p} \leq M$. Then there exists $T^\ast= T^\ast(M, \alpha,p)>0$ and a suitable weak solution $u$ to \eqref{eq: NSalpha} which is smooth on $[T^\ast, \infty)$. Moreover, $T^\ast$ is uniformly bounded away from $\alpha=\frac{3}{4}$.
\end{prop}
%\begin{remark} \end{remark}
\begin{proof} 
Let us consider a global suitable weak solution $u$ of the Navier-Stokes equations of order $\alpha$ obtained as the limit of the regularized system \eqref{eq: regularizedNS}. Let us recall from \cite{TY, Scheffer1, CDM} that there exists $\epsilon=\epsilon(\alpha)>0$ and $\kappa=\kappa(\alpha)>0$ such that if for some $(x_0,t_0)$
\begin{equation}\label{eq: epsilonreg}
 \int_{t_0-1}^{t_0}\int_{B_1(x_0)} \left( \mathcal{M} \lvert u \rvert^2 + \lvert p\rvert\right)^\frac{3}{2} < \epsilon,
\end{equation}
then $u$ is smooth in a neighborhood of $(x_0,t_0)$. Indeed, the maximal operator $\mathcal{M} $ accounts for the non-local effects of the fractional Laplacian and can be removed in the case $\alpha=1$, as in the classical $\epsilon$-regularity of Scheffer. %In the fractional case.

From \cite[Theorem 1.1]{JY}, we infer 
\begin{equation}\label{eq: L2decay}
\norm{u(t)}_{L^2}^2 \leq C(1+t)^{-\frac{3}{2\alpha}(\frac{2}{p}-1)} \, ,
\end{equation}
where $C=C(M, \alpha)>0$. In \cite{JY}, this decay rate is obtained for the Leray-Hopf solution obtained through Garlekin approximation. However, as in \cite{Schonbek} for $\alpha=1$, the argument also applies to the Leray-Hopf solution obtained through the regularized system \eqref{eq: regularizedNS}. Since the decay rate coincides with the one of the fractional heat equation, it cannot, in general, be expected for $u_0 \in L^2$ only. Let $T^\ast:= T^\ast(\alpha, M,p)$ be such that $\norm{u(t)}_{L^2}^2 \leq  \epsilon^\frac{1}{3}$ for $t\in [T^\ast-2^\alpha, + \infty)$.% Recalling \eqref{eq: energyineq2}, this proves the claim.

%Let $(u,p)$ be a suitable weak solution of the Navier-Stokes equations of order $\alpha$ starting from $u_0\in L^2 \cap L^p$ with $\norm{u_0}_{L^2 \cap L^p} \leq M$. By the Gagliardo-Nirenberg-Sobolev inequality and interpolation $u\in L^\frac{6+4\alpha}{3}(\mathbb{R}^3 \times (0,+\infty))$.  
We estimate by H\"older, the maximal function estimate $\norm{\mathcal{M} f}_{L^p} \leq C \norm{f}_{L^p}$ for $1<p\leq + \infty$, Calderon-Zygmund and interpolation
\begin{align}\label{eq: suitsol smallnessreq}
 \int_{t_0-1}^{t_0}\int_{B_1} \big (\mathcal{M} \lvert u \rvert^2 &+ \lvert p \rvert \big )^\frac{3}{2} \, \mathrm{d}x \, \mathrm{d}t 
\leq \left(  \int_{t_0-1}^{t_0}\int_{B_1} \left( \mathcal{M} \lvert u \rvert^2 + \lvert p \rvert\right)^\frac{3+2\alpha}{3} \, \mathrm{d}x \, \mathrm{d}t \right)^\frac{9}{2(3+2\alpha)} %\lvert Q_{2}\rvert^\frac{4\alpha-3}{2(3+2\alpha)}
 \nonumber \\
&\leq C \left(\int_{t_0-1}^{t_0} \int_{\mathbb{R}^3} \lvert u \rvert^\frac{6+4\alpha}{3} \, \mathrm{d}x \, \mathrm{d}t \right)^\frac{9}{2(3+2\alpha)} \nonumber \\
&\leq C \left(\int_{t_0-1}^{t_0} \int_{\mathbb{R}^3} \lvert (-\Delta)^{\alpha/2} u \rvert^2 \, \mathrm{d}x \, \mathrm{d}t \right)^\frac{9}{3+2\alpha} \left( \esssup_{t \in [t_0-1, t_0]} \norm{u(t)}_{L^2}\right)^\frac{6\alpha}{3+2\alpha} \, .
\end{align}
%We now claim that there exists a \textit{particular} global suitable weak solution to \eqref{eq: NSalpha} starting from $u_0$ and a time $T^\ast= T^\ast(M, \alpha, \epsilon)$ such that the right-hand side of \eqref{eq: suitsol smallnessreq} is bounded by $\epsilon$ for any $t_0 \in [T^\ast, +\infty)$. 
This concludes the proof, since from \eqref{eq: epsilonreg} we then infer that $u$ is regular in $[T^\ast, \infty)$.%$Sing(u) \subseteq \mathbb{R}^3 \times [0,T^\ast)$. %In other words, this particular suitable weak solution is a classical solution on $\mathbb{R}^3 \times [T^\ast, +\infty)$ and hence unique among Leray-Hopf solutions by Proposition \ref{prop: weakstrong uniq}. 
\end{proof}

\section{Global-in-time stability}
\begin{prop}[Global-in-time stability]\label{prop:global} Let $\frac{3}{4}< \alpha \leq \frac{5}{4}$, $p\in [1,2)$ and $s\in (\frac{5}{2}-2\alpha, 1]$. Assume that there exists an a priori global smooth solution $u\in C([0,+\infty), H^s)$  to the fractional Navier-Stokes equations of order $\alpha$ with divergence-free initial datum $u_0\in H^s\cap L^p$. If additionally
\begin{equation}\label{eq: add integrability assump}
u\in L^2_{loc}([0,+\infty), H^{s+\alpha+\delta}) \text{ and } Du \in L^2_{loc}([0,+\infty), H^s)
\end{equation}
for some positive $\delta>0$, then there exists $\epsilon>0$ such that for all $v_0\in H^s \cap L^p$ divergence-free and all $\frac{3}{4}< \beta \leq \frac{5}{4}$ satisfying
\begin{equation}\label{hyp}
\norm{u_0-v_0}_{H^s\cap L^p}^2 + \lvert \alpha-\beta \rvert^\delta < \epsilon
\end{equation}
there exists a unique global smooth solution $v\in C([0,+\infty), H^s)$ to the fractional Navier-Stokes equations of order $\beta$ with initial datum $v_0$. 
\end{prop}
\begin{remark}\label{rmk: dropLp} The requirement that $v_0 \in L^p$ for some $p\in [1,2)$ is in fact necessary only for $\alpha \in (\frac 34 , \frac 56)$ to obtain the decay of the $L^2$-norm of the solution in the proof of Proposition~\ref{sec:reg-suit}. Moreover, only boundedness of $v_0$ in $L^p$ (and not closeness to $u_0$) would suffice. 
\end{remark}

\begin{proof} Let $v_0 \in H^s\cap L^p$ divergence-free and $\frac{3}{4}< \beta \leq \frac{5}{4}$. By Corollary \ref{cor: LeraySingTime} and Proposition \ref{prop: eventualreg}, there exists an eventual regularization time $T^\ast(\norm{v_0}_{L^2\cap L^p}, \beta,p)>0$ such that suitable weak solutions to the fractional Navier-Stokes equations of order $\beta$ starting from $v_0 $ are smooth after time $T^\ast$. By Proposition \ref{prop: eventualreg}, we may choose $T^\ast$ uniformly for $\beta$ and $v_0$ verifying \eqref{hyp} for some $0<\epsilon\leq 1$. The conclusion then follows from Proposition \ref{prop: locstab H3} applied with $T=T^\ast$.
\end{proof}

\begin{proof}[Proof of Theorem~\ref{thm:main2}] We check that the assumptions \eqref{eq: add integrability assump} of Proposition~\ref{prop:global} are satisfied for $s=1$. Let $u\in C([0,+\infty), H^1)$ be an a priori global smooth solution with initial datum $u_0$. Fix $T>0$. Then, differentiating the equation $k$-times and performing energy estimates, we see that $u \in L^\infty((0,T),H^k)\cap L^2((0,T),H^{k+\alpha})$ for every $k\geq 1$. Hence, if $u_0 \in H^s$ then $u\in L^2_{loc}([0,+\infty), H^r)$ for any $1 \leq r \leq s+\alpha$. We infer that in the ipodissipative case, the additional integrability assumption \eqref{eq: add integrability assump} is fulfilled with $\delta:=1-\alpha$ provided $u_0\in H^{2-\alpha}$. Since $2-\alpha \leq 2-\frac{3}{4}=\frac{5}{4}$ for $\alpha \in (\frac{3}{4},1)$, it is enough to ask $u_0 \in H^s$ for some $s\geq \frac{5}{4}$. Finally, in the hyperdissipative case, \eqref{eq: add integrability assump} is fulfilled for any $s>1$.
\end{proof}

\begin{remark}
In the hyperdissipative range $\alpha>1$, for any $s\in (0, \frac{1}{2}]$, we could adapt the proof of Theorem \ref{thm:main2} to deduce from  Proposition~\ref{prop:global} - for instance - openness of initial data and fractional orders giving rise to global smooth solutions in 
$$\left\{ u_0 \in  H^s(\mathbb{R}^3; \mathbb{R}^3) : \div u_0=0\right\} \times \Big(\frac{5}{4}-\frac{s}{2}, \frac{5}{4}\Big]\, .$$
\end{remark}

\begin{proof}[Proof of Theorem~\ref{thm:main1}]
Let $u_0\in H^\delta$ divergence-free with $\norm{u_0}_{H^\delta} \leq M$ and let $u$ be the unique smooth Leray-Hopf solution to the fractional Navier-Stokes equations of order $\alpha=\frac{5}{4}$ starting from $u_0$. We claim the following a priori estimate
\begin{equation}\label{eq: aprioriest at 54}
\frac{\mathrm{d}}{\mathrm{d}t} \norm{u(t)}_{H^\delta}^2 + \norm{u(\tau)}_{H^{\delta+\alpha}}^2 \leq C_A \norm{(-\Delta)^{\alpha/2} u(t)}_{L^2}^2 \norm{u(t)}_{H^\delta}^2 \, ,
\end{equation}
where $C_A>0$ is independent of $\delta$. Indeed, arguing as in the proof of Lemma \ref{lem: energyineq for diffeq H1}, we obtain the following energy inequality for derivatives of order $\delta$ 
\begin{align*}
\frac{1}{2}\frac{\mathrm{d}}{\mathrm{d}t} \norm{ (-\Delta)^{\delta/2} u(t)}_{L^2}^2 &+ \norm{ (-\Delta)^{(\alpha+\delta)/2} u(t)}_{L^2}^2 \leq \left \lvert \int (-\Delta)^{\delta/2}\left( u \cdot \nabla) u \right) \cdot (-\Delta)^{\delta/2} u  \, \mathrm{d}x\right \rvert \\
&\leq (1+\bar D) \bar C ^2  \norm{(-\Delta)^{(\delta+\alpha)/2}u}_{L^2} \norm{u}_{\dot{H}^{\frac{5}{2}-\alpha}} \norm{(-\Delta)^{\delta/2} u}_{L^2}
\end{align*}
%where $D=D(p_1, p_2)$ is the constant of Theorem A.8 in \cite{KPV} 
by \eqref{eqn:commuta} (applied with $p_1=\frac{6}{3-2\alpha}$, $p_2=\frac{6}{2\alpha}$, $s_1=s$, $s_2=0$). In particular, it does not depend on $\delta$. Observing that at $\alpha=\frac{5}{4}$ it holds $ \norm{u}_{\dot{H}^{\frac{5}{2}-\alpha}} = \norm{(-\Delta)^{\alpha/2} u}_{L^2}$, this proves \eqref{eq: aprioriest at 54} with $C_A = 2 (1+\bar D)^2 \bar C^4 \leq 8 \bar D^2 \bar C^4$. From \eqref{eq: aprioriest at 54}, we deduce by Gr\"onwall the estimates
\begin{align*}
\norm{u}_{L^\infty([0, T], H^\delta)}^2 &\leq \norm{u_0}_{H^\delta}^2 e^{C_A \int_0^T \int \lvert (-\Delta)^{\alpha/2} u(x,t) \rvert^2 \, \mathrm{d}x\, \mathrm{d}t} \leq M^2 e^{C_A M^2} \, ,\\
\int_0^T \norm{u(t)}_{H^{\delta+\frac{5}{4}}}^2 \, \mathrm{d}t &\leq \norm{u_0}_{H^\delta}^2 + \int_0^T \norm{(-\Delta)^{\alpha/2} u(t)}_{L^2}^2 \norm{u(t)}_{H^\delta}^2 \, \mathrm{d}t \leq M^2(1+M^4 e^{C_A M^2} )\, .
\end{align*}
In particular, $u$ satisfies the additional integrability assumption \eqref{eq: add integrability assump}. Thus, by Proposition \ref{prop:global}, there exists $\epsilon>0$ such that for any fractional order $\beta \in [1, \frac{5}{4}]$ satisfying $\lvert \frac{5}{4}-\beta \rvert < \epsilon$ there exists a unique global, smooth solution to the fractional Navier-Stokes equations of order $\beta$ starting from $u_0$. Notice that the additional assumption $u_0 \in L^p$ for some $p \in [1,2)$ can be dropped thanks to Remark \ref{rmk: dropLp}. We recall from the explicit choice of $\epsilon$ in \eqref{eq: localchoiceofepsilon} that 
\begin{equation}\label{eq: epsilonexplicit 54}
\epsilon  \leq \min \left \{ \left( \frac{\delta}{4} \right)^{\frac{1}{\delta}}, \left((C_0 T^\ast)^\frac{4\beta -5 + 2\delta}{2(\beta+\delta)} %\left( \frac{2(\beta+\delta)}{4\beta-5+2\delta}\right)^\frac{4\beta-5+2\delta}{2(\beta+\delta)} 
e^{2C_1 \norm{u}_{L^2([0,T^\ast], H^{\frac{5}{4}+\delta})}^2} \max\{C_2  \norm{u}_{L^2([0,T^\ast], H^{\frac{5}{4}+\delta})}^2,1\}\right)^{-\frac{1}{\delta}} \right \}  \,,
\end{equation}
where now $T^\ast$ is the eventual regularization time given by Corollary \ref{cor: LeraySingTime} and the constants $C_0$, $C_1$ and $C_2$ are coming from Lemma \ref{lem: energyineq for diffeq H1} applied with $s=\delta$. In the hyperdissipative range, we have an explicit upper bound on $T^\ast$ in terms of $\beta$ and $M$ through \eqref{eq: upperbound for Tstar hyp}. Indeed, it follows from \eqref{eq: upperbound for Tstar hyp} by a simple computation that $T^\ast$ can be bounded, uniformly in $\beta$, by
\begin{equation}
T^\ast \leq e^{M^\frac{5}{2} \bar C^3/e} \, ,
\end{equation}
Moreover, from the explicit expression of $C_0$ and $C_1$ in Lemma \ref{lem: energyineq for diffeq H1} (see \eqref{eq: C0Hs}), we have infer

%\begin{align*}
%C_0^\frac{4\beta}{4\beta-5+2\delta} \leq (24(1+D)\bar C)^2 \text{ and } C_1 \leq 30(1+D)^2 \bar C^6
%\end{align*}
%where $$D:= \max \left \{ \sup_{\beta \in [1, \frac{5}{4}]} D\left(\frac{6}{2\beta}, 2\right), \sup_{\beta \in [1, \frac{5}{4}], \delta \in (0,1]} D\left(\frac{6}{2(s+\beta)-2}, \frac{6}{5-2(s+\beta)} \right) \right \}$$ is a universal constant (recall that we denote by $D(p_1, p_2)$ the constant of Theorem A.8 from \cite{KPV}, which does not depend on $s_1$ and $s_2$). 
\begin{align*}
C_0^\frac{4\beta -5 + 2\delta}{2(\beta+\delta)}  \leq 24(1+ \bar D)\bar C^2 \leq 48 \bar D \bar C^2 \text{ and } C_1 \leq 6(3(1+\bar D)^2\bar C^2+2 \bar C^6) \leq 72 \bar D^2 C^6
\end{align*}
Hence we can estimate the $\epsilon$ in terms of $M$ and $\delta$ by 
%\begin{equation}\label{eq: epsilonexplicit 54 ugly}
%\left( (24(1+D) \bar C^2) \max \{ 1, e^{M^\frac{5}{2} \bar C^3/e}\} \max\{2 C_2 M^2(1+M^4 e^{C_A M^2}), 1 \} e^{60(1+D)^2 \bar C^6 M^2(1+M^4 e^{C_A M^2})}\right)^{-\frac{1}{\delta}} \,,
%\end{equation}
\begin{equation}\label{eq: epsilonexplicit 54 ugly}
\epsilon := \min \left \{ \left(\frac{\delta}{4}\right)^{\frac{1}{\delta}}, \left( 48 \bar D \bar C^2 e^{M^\frac{5}{2} \bar C^3/e} \max\{2 C_2 M^2(1+M^4 e^{C_A M^2}), 1 \} e^{144 \bar D^2 \bar C^6 M^2(1+M^4 e^{C_A M^2})} \right)^{-\frac{1}{\delta}} \right \}
\end{equation}
where $C_2$ is the universal constant from Lemma \ref{lem: fraclapRn}.
\end{proof}
%\begin{remark} The bound \eqref{eq: epsilonexplicit 54 ugly} is by no means optimal. All constants being explicit, it can be computed explicitly given $\delta$ and $M$. %{\color{red} add comments on the H1 norm in theo1, it can be weakened}
%\end{remark}

\textbf{ Acknowledgements}. The authors were partially supported by
the Swiss national science foundation grant 182565 "Regularity issues for the Navier-Stokes equations and for other variational problems".


\begin{thebibliography}{10}

\bibitem{BMR}
D.~Barbato, F.~Morandin, and M.~Romito.
\newblock Global regularity for a slightly supercritical hyperdissipative
  {N}avier-{S}tokes system.
\newblock {\em Anal. PDE}, 7(8):2009--2027, 2014.

\bibitem{BM}
A.~Majda and A.~Bertozzi.
\newblock Vorticity and incompressible flow.
\newblock {\em Cambridge Texts in Applied Mathematics, 27}, Cambridge University Press, 2002.

\bibitem{CKN}
L.~Caffarelli, R.~Kohn, and L.~Nirenberg.
\newblock Partial regularity of suitable weak solutions of the
  {N}avier-{S}tokes equations.
\newblock {\em Comm. Pure Appl. Math.}, 35(6):771--831, 1982.

\bibitem{CS}
L.~Caffarelli and L.~Silvestre.
\newblock An extension problem related to the fractional {L}aplacian.
\newblock {\em Comm. Partial Differential Equations}, 32(7-9):1245--1260, 2007.

\bibitem{CaffarelliVasseur}
L.~Caffarelli and A.~ Vasseur.
\newblock Drift diffusion equations with fractional diffusion and the quasi-geostrophic equation.
\newblock {\em Ann. of Math.}, 171(2): 1903--1930, 2010.

\bibitem{CDM}
M.~Colombo, C.~De Lellis and A.~Massaccesi.
\newblock The generalized {C}affarelli-{K}ohn-{N}irenberg {T}heorem for the hyperdissipative {N}avier-{S}tokes system.
\newblock {\em ArXiv e-prints}, December 2017.

\bibitem{CH}
M.~Colombo and S.~Haffter.
\newblock Global regularity for the nonlinear wave equation with slightly supercritical power.
\newblock {\em ArXiv e-prints}, November 2019.

\bibitem{Cotivicol}
M.~Coti~Zelati and V.~Vicol.
\newblock On the global regularity for the supercritical {SQG} equation.
\newblock {\em Indiana Univ. Math. J.}, 65(2):535--552, 2016.

\bibitem{Dancona}
P.~D'Ancona.
\newblock A Short Proof of Commutator Estimates.
\newblock {\em J. Fourier. Anal. Appl.}, 142:1--13, 2018.


\bibitem{Hopf}
E.~Hopf.
\newblock \"Uber die {A}nfangswertaufgabe f\"ur die hydrodynamischen {G}rundgleichungen.
\newblock{\em Math. Nachr.}, 4:213--231, 1951.

\bibitem{JiuWang}
Q.~Jiu and Y.~Wang.
\newblock On possible time singular points and eventual regularity of weak
  solutions to the fractional {N}avier-{S}tokes equations.
\newblock {\em Dyn. Partial Differ. Equ.}, 11(4):321--343, 2014.

\bibitem{JY}
Q.~Jiu and H.~Yu.
\newblock {\em Decay of solutions to the three-dimensional generalized Navier–Stokes equations}.
\newblock {\em Asymptotic Analysis}, 94(1):105--124, 2015.

%\bibitem{KP}
%N.~H. Katz and N.~Pavlovi\'c.
%\newblock A cheap {C}affarelli-{K}ohn-{N}irenberg inequality for the
%  {N}avier-{S}tokes equation with hyper-dissipation.
%\newblock {\em Geom. Funct. Anal.}, 12(2):355--379, 2002.

\bibitem{KPV}
C.~Kenig, G.~Ponce and L.~Vega.
\newblock Well-posedness and scattering results for the generalized {K}orteweg-de {V}ries equation via the contraction principle.
\newblock {\em Comm. Pure Appl. Math.}, 46(4):527--620, 1993.

\bibitem{Leray}
J.~Leray.
\newblock Sur le mouvement d'un liquide visqueux emplissant l'espace.
\newblock {\em Acta Math.}, 63(1):193--248, 1934.

\bibitem{Lions}
J.-L. Lions.
\newblock {\em Quelques m\'ethodes de r\'esolution des probl\`emes aux limites
  non lin\'eaires}.
\newblock Dunod; Gauthier-Villars, Paris, 1969.


\bibitem{Scheffer1}
V.~Scheffer.
\newblock Partial regularity of solutions to the {N}avier-{S}tokes equations.
\newblock {\em Pacific J. Math.}, 66(2):535--552, 1976.

\bibitem{Scheffer2}
V.~Scheffer.
\newblock Hausdorff measure and the {N}avier-{S}tokes equations.
\newblock {\em Comm. Math. Phys.}, 55(2):97--112, 1977.

\bibitem{Silvestre}
L.~Silvestre.
\newblock Eventual regularization for the slightly supercritical quasi-geostrophic equation.
\newblock {\em Ann. Inst. H. Poincar\'{e} Anal. Non Lin\'{e}aire}, 27(2):693-704, 2010.

\bibitem{SGS}
H.~Shang, Y.~Guo and M.~Song.
\newblock Global regularity for the supercritical active scalars.
\newblock {\em Z. Angew. Math. Phys.}, 68(3):Art. 64, 15, 2018.


\bibitem{Schonbek}
M.E.~Schonbek.
\newblock Large time behaviour of solutions to the {N}avier-{S}tokes equations.
\newblock {\em Comm. Partial Differential Equations}, 11(7):733--763, 1986. 

\bibitem{Stein}
E.~Stein.
\newblock Singular integrals and differentiability properties of functions.
\newblock {\em Princeton Mathematical Series, No. 30}, Princeton University Press, 1970.

\bibitem{TY}
L.~Tang and Y.~Yu.
\newblock Partial regularity of suitable weak solutions to the fractional
  {N}avier-{S}tokes equations.
\newblock {\em Comm. Math. Phys.}, 334(3):1455--1482, 2015.

\bibitem{Tao}
T.~Tao.
\newblock Global regularity for a logarithmically supercritical
  hyperdissipative {N}avier-{S}tokes equation.
\newblock {\em Anal. PDE}, 2(3):361--366, 2009.

%\bibitem{Yang}
%R.~{Yang}.
%\newblock {On higher order extensions for the fractional Laplacian}.
%\newblock {\em ArXiv e-prints}, February 2013.

\end{thebibliography}
\end{document}